\documentclass[11pt, leqno]{article}

\usepackage{enumitem}
\usepackage{amsmath,amsfonts,amsthm,amssymb,amsxtra,setspace,xspace,graphicx,lmodern,mathrsfs}
\usepackage{tabularx}

\usepackage{enumitem}
\usepackage{graphicx}
\usepackage{color}
\usepackage{hyperref}
\usepackage{nameref}
\usepackage{pgf,tikz}
\usepackage{pdfsync}

\usetikzlibrary{arrows}

\usepackage[inner=2.5cm,outer=4cm,bottom=4cm,top=3.5cm]{geometry}

\newtheorem{theorem}{Theorem}[section]
\newtheorem{proposition}[theorem]{Proposition}

\newtheorem{defn}[theorem]{Definition}
\newtheorem{lemma}[theorem]{Lemma}

\newtheorem{remark}[theorem]{Remark}
\newcommand\R{\mathbb{R}}
\newcommand\RN{\mathbb{R}^N}

\newcommand\X{\mathcal{X}_p}
\numberwithin{equation}{section}

\newcommand{\nrm}[2]{\left\|{#1}\right\|_{#2}}

\begin{document}

\title{\vspace{-2cm}\bf The Cauchy problem for the\\ fast $p-$Laplacian evolution equation.\\
Characterization of the global Harnack principle\\ and fine asymptotic behaviour}

\author{\Large Matteo Bonforte,$^{\,a}$
 Nikita Simonov$^{\,b}$ and Diana Stan$^{\,c}$\\[3mm]
}
\date{}

\maketitle

\begin{abstract}
We study fine global properties of nonnegative solutions to the Cauchy Problem for the fast $p$-Laplacian evolution equation $u_t=\Delta_p u$ on the whole Euclidean space, in the so-called ``good fast diffusion range'' $\tfrac{2N}{N+1}<p<2$. It is well-known that non-negative solutions behave for large times as $\mathcal{B}$, the Barenblatt (or fundamental) solution, which has an explicit expression.

We prove the so-called Global Harnack Principle (GHP), that is, precise global pointwise upper and lower estimates of nonnegative solutions in terms of $\mathcal{B}$. This can be considered the nonlinear counterpart of the celebrated Gaussian estimates for the linear heat equation.

We characterize the maximal (hence optimal) class of initial data  such that the GHP holds, by means of an integral tail condition, easy to check. The GHP is then used as a tool to analyze the fine asymptotic behavior for large times. For initial data that satisfy the same integral condition, we prove that  the corresponding solutions behave like the Barenblatt with the same mass, uniformly in relative error.

When the integral tail condition is not satisfied we show that both the GHP and the uniform convergence in relative error, do not hold anymore, and we provide also explicit counterexamples. We then prove a ``generalized GHP'', that is, pointwise upper and lower bounds in terms of explicit  profiles with a tail different from $\mathcal{B}$. Finally, we derive sharp global quantitative upper bounds of the modulus of the gradient of the solution, and, when data are radially decreasing, we show uniform convergence in relative error for the gradients.

To the best of our knowledge, analogous issues for the linear heat equation $p=2$, do not possess such clear answers, only partial results are known.

\end{abstract}

principle; Gradient decay estimates; Tail behaviour; Asymptotic behaviour.\\

\vfill
\begin{itemize}[leftmargin=*]\itemsep2pt \parskip3pt \parsep0pt
\item[(a)] Departamento de Matem\'{a}ticas, Universidad Aut\'{o}noma de Madrid,\\
ICMAT - Instituto de Ciencias Matem\'{a}ticas, CSIC-UAM-UC3M-UCM, \\
Campus de Cantoblanco, 28049 Madrid, Spain.\\
E-mail:\texttt{~matteo.bonforte@uam.es }\;\;
Web-page: \url{http://verso.mat.uam.es/~matteo.bonforte}

\item[(b)]  Ceremade, UMR CNRS n$^\circ$~7534, Universit\'e Paris-Dauphine, PSL Research University,\par
Place de Lattre de Tassigny, 75775 Paris Cedex~16, France. \par
E-mail:\texttt{~simonov@ceremade.dauphine.fr}\;\;
Web-page: \url{https://sites.google.com/view/simonovnikita/}

\item[(c)] Department of Mathematics, Statistics and Computation, University of Cantabria, \\
Avd. Los Castros 44, 39005 Santander, Spain.\\
E-mail:\texttt{~diana.stan@unican.es}\;\;
Web-page: \url{https://personales.unican.es/stand/}

\end{itemize}

\newpage

\tableofcontents

\normalsize

\vfill

\medskip

\noindent\textbf{Acknowledgments. }M.B. and N.S. were partially supported by the Project MTM2017-85757-P (Ministry of Science and Innovation, Spain).  M.B. acknowledges financial support from the Spanish Ministry of Science and Innovation, through the ``Severo Ochoa Programme for Centres of Excellence in R\&D'' (CEX2019-000904-S) and by the E.U. H2020 MSCA programme, grant agreement 777822.
N.S. was partially funded by the FPI-grant BES-2015-072962, associated to the project MTM2014-52240-P (Ministry of Science and Innovation, Spain). N.S. has been partially supported by the Project EFI (ANR-17-CE40-0030) of the French National Research Agency (ANR). This work was supported by grants from R\'egion Ile-de-France.
D.S.  has been supported by PGC2018-094522-B-I00 from the MICINN of the Spanish Government and by the project VP42 ``Ecuaciones de evoluci\'on no
lineales y no locales y aplicaciones'' from the Cons. de Univ., Igualdad, Cultura y Deporte, Cantabria, Spain.

\smallskip\noindent {\sl\small\copyright~2021 by the authors. This paper may be reproduced, in its entirety, for non-commercial purposes.}

\newpage

\section{Introduction and main results}

We consider nonnegative solutions to the Cauchy problem for the \textit{$p$-Laplacian Evolution equation }(PLE) posed on $\RN\times (0,\infty)$, with $N\ge 1$:
\[
u_{t}  = \Delta_p u := \mathrm{div}(|\nabla u|^{p-2}\nabla u)\,,  \tag{PLE}\label{PLE}
\]
where  $\Delta_p u$ denotes the well-known $p$-Laplacian operator with $p\ge 1$.
This nonlinear evolution equation appeared in the mathematical description  of compressible fluid flows in a homogeneous isotropic rigid porous medium, \cite{Ladyz1964}. Since then, it has been the basic model in many applications for instance,  in image reconstruction \cite{AndreuCasellesMazon-Book,BarVaz} and in game theory (tug-of-war games), \cite{Peres-Sheffield,
Manfredi1}. The \ref{PLE} has been widely investigated since the early 1960's also   because of its own mathematical interest, beauty and difficulty, being a prototype of a nonlinear evolution equation with gradient-dependent diffusivity, possibly degenerate or singular, see \cite{DiBenedetto}, \cite[Section 11]{JLVSmoothing} and references therein.
We will consider the Cauchy problem associated to the \eqref{PLE}:
\begin{equation}\label{PLEcauchy}
  \left\{ \begin{array}{ll}
  u_{t}(x,t) = \Delta_p u(x,t) &\qquad\text{for }t>0 \text{ and } x \in \RN, \\
  u(x,0)  =u_0(x) &\qquad\text{for } x \in \RN,
    \end{array}
    \right.
\end{equation}
in the so-called \textit{good fast diffusion range} and for nonnegative integrable initial data, that is
\begin{equation}\label{Hyp.GFDE}
1\le p_c:=\frac{2N}{N+1}<p<2\qquad\mbox{and}\qquad 0\le u_0\in L^1(\RN)\,.
\end{equation}
The basic theory of existence, uniqueness, boundedness and regularity is well understood, see Subsection \ref{ssec.Basic.Theory}.
Let us briefly point out the differences that occur in the distinct ranges and values of the parameter $p\ge 1$, when dealing with (nonnegative) solutions corresponding to nonnegative integrable initial data. The degenerate or slow diffusion case corresponds to $p>2$: mass is preserved along the evolution, there is finite speed of propagation (compactly supported initial data generate solutions with compact support for all times), and nonnegative integrable data give bounded solutions which are positive inside their support.
The case $p=2$ corresponds to the linear Heat Equation (HE), with infinite speed of propagation and $C^\infty$ smooth solutions obtained by the Gaussian representation formula.
The case under consideration in this paper is the singular or fast diffusion case, which corresponds to $p\in (1,2)$, where infinite speed of propagation holds (compactly supported initial data produce everywhere-positive solutions). Mass is preserved only in the good fast diffusion range, namely when $p\in (p_c,2)$, and, in this case, solutions are positive and bounded everywhere, for all $t>0$. In the very fast diffusion regime, namely when $p\in (1, p_c)$ there is a regularity breakdown, mass is not preserved and solutions extinguish in finite time \cite{BIV,DiBenedettoHerrero,DGV-Book}. It is worth mentioning the very singular case $p=1$, the Total Variation Flow, \cite{AndreuCasellesMazon-Book, BF2012} because of its important applications in image processing. In this latter case, even bounded solutions may be discontinuous. The goal of this paper is to describe the behaviour of solutions in the most precise possible way, under assumptions \eqref{Hyp.GFDE}.

When $p>p_c$, the large time asymptotic behavior is described in terms of self-similar solutions conserving mass,  known as \textit{Barenblatt (or fundamental) solutions}, which have as initial datum $M\delta_0$, the Dirac delta centered at the origin  with mass $M$, see \cite{KaminVazquez1988} when $p>2$ and Section \ref{Section:asymptotics} for $p\in (p_c,2)$.   They are also the key tool to understand the finite and infinite speed of propagation, and have an explicit formula
\begin{equation}\label{Berenblatt.000}
\mathcal{B}(x,t;M)= t^{\frac{1}{2-p}} \left[ b_1  t^{\frac{\beta p}{p-1}} M^{\frac{\beta p(p-2)}{p-1}} + b_2\left|x\right|^{\frac{p}{p-1}}\right]_+^{-{\frac{p-1}{2-p}}}
\end{equation}
See formula \eqref{parameters} and \eqref{BarenblattMass} for explicit expression of $\alpha,\beta, b_1,b_2$.  Note than when $p\to 2$ the Barenblatt solution converges pointwise to the Gaussian, the fundamental solution of the HE.  For the present discussion we just need to recall that $b_2\sim 2-p$, or just that $b_2$ is negative in the slow diffusion case, $p>2$, exhibiting finite speed of propagation. On the other hand, when $p<2$, the support of $\mathcal{B}$ is immediately spread on the whole $\RN$, showing infinite speed of propagation, as in the classical case of the linear HE, $p=2$.  It is worth noticing that when $p<2$, the tail of the fundamental solution is ``fat'', of polynomial type, contrary to exponential decay in the Gaussian case. Roughly speaking, as $p$ becomes smaller, the tail of the fundamental solution becomes fatter, and much more mass can escape to infinity, showing that the diffusion is faster as $p$ decreases.
We shall address the following question(s):
\begin{equation}\label{Q}\tag{Q}
\begin{array}{cc}
    \textit{Do nonnegative integrable solutions behave like the fundamental solution?}\\[2mm]
  \textit{If yes, in which sense? Do they have the same tail behaviour?}
\end{array}
\end{equation}

The main goal of this paper is to answer question \eqref{Q} in the sharpest possible way. When the initial datum satisfies appropriate tail conditions, we derive quantitative and explicit upper and lower bounds for the solution to Problem \eqref{PLEcauchy} in terms of Barenblatt solutions, which we call \textit{Global Harnack Principle} (GHP). Surprisingly enough we are able to characterize the optimal class of nonnegative integrable data that produce solutions that satisfy the GHP. The second issue that we address, concerns the asymptotic behaviour: we know that nonnegative solutions behave for large times as the Barenblatt with the same mass, in the strong $L^1$ topology. The same question with a finer topology has a completely different answer:  uniform convergence in relative error towards the Barenblatt is possible if and only if GHP holds, i.e. if the initial datum satisfies our ``tail condition''. We also address the same questions for the (modulus of the) gradient of the solutions and we obtain a sharp pointwise decay. Convergence in relative error for the gradient is proven for radial solutions, a first step towards understanding a really delicate issue.

\subsection{Main result 1. Global Harnack Principle and convergence in relative error}
In order to measure in an optimal way the tail behaviour of integrable functions, it is convenient to define the following integral quantity:
\[
\|f\|_{\X}:=\sup_{R>0} R^{\frac{p}{2-p}-N} \int_{\RN\setminus{B_R(0)}}  |f(x)| dx\,.
\]
This quantity defines a natural norm on the following subspace of $L^1(\RN)$

\begin{equation}\label{X}
\X=\left\{  f\in L^1(\RN)\,:\,  \quad \|f\|_{\X}< +\infty \right\}.
\end{equation}
It is clear that $C^\infty_c(\RN)\subset \X$, hence $\X$ is dense in $L^1(\RN)$ in the strong $L^1$-topology. On the other hand, $(\X,\|\cdot\|_{\X}+\|\cdot\|_{L^1(\RN)})$ is a Banach space, see \cite{Simonov:thesis,BS2020} and Section \ref{Section:Counterexample} for more details.

Indeed, $\X$ is somehow a ``natural space'' in our context, as we shall explain in Remark \ref{Rem.THM1}.3.

The Global Harnack Principle  for the Cauchy problem was first proven in \cite{BV2006,JLVSmoothing} for the Fast Diffusion Equation (in the so-called good-range, where mass is still preserved)
\begin{equation}\label{FDE}\tag{FDE}
u_t=\Delta u^m, \qquad\mbox{with}\qquad m\in \left(\tfrac{N-2}{N},1\right),
\end{equation}
under a non-sharp pointwise tail condition, which in this context becomes:
\begin{equation}\label{tail:space}
\mathcal{A}_p=\left\{ f\in L^1(\RN)\,:\, \exists A,R_0>0 \quad |f(x)|\le A |x|^{-\frac{p}{2-p}}\quad \text{ for all }|x|\ge R_0 \right\}.
\end{equation}
The Barenblatt solutions in the good fast diffusion range, have the form (with $\alpha,\beta, b_1,b_2>0$  given in \eqref{parameters})
\begin{equation}\label{BarenblattMass}
\mathcal{B}(x,t;M)= t^{\frac{1}{2-p}} \left[ b_1 t^{\frac{\beta p}{p-1}} M^{\frac{p-2}{p-1}\beta p} + b_2\left|x\right|^{\frac{p}{p-1}}\right]^{{-\frac{p-1}{2-p}}}.
\end{equation}
It is clear that $\mathcal{B}(x,t;M)\in\mathcal{A}_p$, hence $\mathcal{B}(x,t;M)\in\X$ for all $t>0$. Moreover, any function $f$ with a lower tail than the Barenblatt belongs to $\mathcal{A}_p$, hence to $\X$: there exist constants $A',M,R_0>0$ such that
\[
|f(x)|\le A' \mathcal{B}(x,1;M)\qquad\mbox{for all }|x|\ge R_0\,.
\]
However, in Section \ref{Section:tail} we will prove that $\mathcal{A}_p \subsetneq \X$ by constructing explicit examples of functions $f\in \X\setminus \mathcal{A}_p$.

We are now in the position to state our main result, which completely answers \eqref{Q}.

\begin{theorem}
\label{ghp.thm}
Let $N\ge 1$ and $p_c:=\frac{2N}{N+1}<p<2.$ Let $u$ be a weak solution to Problem \eqref{PLEcauchy} corresponding to an initial datum  $0\le u_0\in L^1(\RN)$. Then, the following statements are equivalent:

\noindent{\rm (i- Characterization in terms of the space $\X$)}
\begin{equation*}
u_0\in\X\setminus\{0\}\qquad\mbox{that is}\qquad 0<\sup_{R>0} R^{\frac{p}{2-p}-N} \int_{\RN\setminus{B_R(0)}}  |u_0(y)| dy<+\infty
\end{equation*}
\noindent{\rm (ii- Global Harnack Principle). }
For any $t_0 >0$, there exist (explicit) constants $\tau_1, \, M_1, \, \tau_2, \, M_2$
such that the following upper and lower bounds hold true for all $x\in \RN$ and $t>t_0$
\begin{equation}\label{ghp.inq}
  \mathcal{B}(x,t-\tau_1;M_1)\, \le \,u(x,t) \,\le \, \mathcal{B}(x,t+\tau_2;M_2)\,.
  \end{equation}
\noindent{\rm (iii- Uniform Convergence in Relative Error) }We have that
\begin{equation}\label{convergence.relative.error.limit}
\lim\limits_{t\rightarrow \infty} \Big\|\frac{u(\cdot,t)}{\mathcal{B}(\cdot,t; M)}-1\Big\|_{L^\infty(\RN)}= 0\,,\quad\mbox{where}\quad M=\|u_0\|_{L^1(\R^N)}\,.
\end{equation}
\end{theorem}
\begin{remark}\label{Rem.THM1}\rm
\begin{enumerate}[leftmargin=*]\itemsep0pt
\item The proof of this Theorem occupies a large part of the paper and consists of several independent results, with their own interest. In Section~\ref{sec:upper.lower} we prove that $i)$ implies $ii)$ while in Section~\ref{Section:asymptotics} the implication $i)\Rightarrow iii)$ is proven. The final equivalence is proven in Section~\ref{Section:Counterexample} where some properties of the space $\X$ are also analyzed. More details are given in Section~\ref{Section:Literature}.
\item As already mentioned in the introduction, we exhibit in Proposition~\ref{JLV.condition.equivalence} another condition to characterize the space $\X$:
    \begin{equation}\label{equiv.cond.Xp.0000}
    f\in\X\quad\mbox{if and only if \quad $f\in L^1(\RN)$ and } \int_{B_{|x|/2}(x)}|f(y)| dy = \mathrm{O}\left(|x|^{N-\frac{p}{2-p}}\right)\,\quad\mbox{as}\,\,|x|\rightarrow +\infty.
    \end{equation}
    A similar condition has been introduced by V\'{a}zquez in~\cite{Vaz2003} as a sufficient condition for the Global Harnack Principle to hold in the case of the \eqref{FDE}.  Lately, in~\cite{BS2020}, a similar condition has been used to prove analogous results in the case of \eqref{FDE} with Caffarelli-Kohn-Nirenberg (CKN) weights, providing new and sharp characterization also in the non-weighted case.
\item The decay assumption on the initial data, i.e. $u_0\in\X$ is sharp, indeed if $u_0\not\in\X$, both GHP and uniform convergence in relative error fail. This shows that somehow \textit{$\X$ is a ``natural space'' }. First, it turns out to be  invariant (or stable) under the $p$-Laplacian flow, meaning that  $u_0\in \X$ if and only if $u(t)\in \X$ for all $t>0$, see Section~\ref{Section:tail}.  Second, as far as strong tail-stability is concerned, the so-called Global Harnack Principle (GHP) says that data in $\X$ produce solutions sandwiched between two Barenblatt solutions, and again, the viceversa turns out to be true. The third aspect concerns the fine asymptotic behaviour: on one hand, it is always true that nonnegative integrable solutions behave like the Barenblatt for large times in the strong $L^1$-topology. On the other hand, some of those  solutions may not converge to the Barenblatt in a finer topology, uniformly in relative error. In other words, the asymptotic behaviour of the tail will not be the same as the Barenblatt solutions in some cases. We remarkably succeed in characterizing by means of simple integral conditions the basin of attraction of the Barenblatt solutions in the topology of uniform convergence in relative error. Such basin of attraction turns out to be  $\X$.
\item  What happens when the initial datum is in $\X^{\bf c}=L^1(\RN)\setminus\X$? In Section \ref{Section:DataNotX} we construct explicit examples of sub-solutions and super-solutions to the \eqref{PLE} equation, with a tail which is slightly fatter that the Barenblatt (somehow the maximal tail in $\X$), and show that they will never satisfy a GHP
    \[
    \frac{1}{(1+|x|)^{\frac{p}{2-p}-\delta}}\lesssim u_0(x)\lesssim \frac{1}{(1+|x|)^{\frac{p}{2-p}-\varepsilon}}
    \qquad\mbox{implies}\qquad
    \frac{c_0(t)}{(1+|x|)^{\frac{p}{2-p}-\delta}}\lesssim u(x,t)\lesssim \frac{c_1(t)}{(1+|x|)^{\frac{p}{2-p}-\varepsilon}}
    \]
    for suitable explicit functions $c_0,c_1$ and suitably small $\varepsilon,\delta>0$, see Section \ref{Section:DataNotX}.
\item  Theorem \ref{ghp.thm} does not hold for sing-changing solutions, as observed for the \eqref{FDE} by V\'azquez
in~\cite[Proposition 18.35]{Vazquez2007}. In that proposition, a 1D solution which exhibits sing changing for any time $t>0$, is constructed, in spite of being generated by compactly supported initial data with positive mass. In higher dimensions, we were only able to find examples on subdomains of $\RN$, constructed by King~\cite{King_1990}.

\item\textit{Some applications. }The GHP has several important applications. Besides the uniform convergence in relative error, equivalent to GHP by the above theorem, it is also fundamental to establish convergence rates, see for instance \cite{ABC, delPinoDolbeaultDNLE}, where (a stricter condition than) GHP is taken as an assumption to make all the machinery work, following \cite{BBDGV}. More details shall be given below. It is also the key to show quantitative stability results in Gagliardo-Nirenberg-Sobolev inequalities, \cite{BDNS2020,BDNS2020Suplement}. Another application concerns reaction-diffusion equations: the GHP is essential to describe the behavior of solutions to reaction-diffusion problems, see for instance \cite{AudritoVaz2017-NTMA} for the doubly nonlinear reaction-diffusion equations and \cite{SV-KPP} for the fractional diffusion counterpart.\vspace{-4mm}
\end{enumerate}\end{remark}
\subsubsection{Relevance of the results}
To the best of our knowledge, in the linear case $p=2$, the closest result to our GHP, is represented by the celebrated Harnack type estimates, often called Gaussian estimates,
\[
 u(t_1,x)\le u(t_2,y)\left(\frac{t_2}{t_1}\right)^{N/2}e^{\frac{|x-y|^2}{4(t_2-t_1)}}\qquad\mbox{for any }x,y\in \RN,~ t_2>t_1>0\,,
\]
proven by Pini \cite{Pi54} and Hadamard \cite{Ha54}, then generalized to uniformly parabolic equations by Moser \cite{Mo64}, Aronson \cite{Ar68} and many others.
However, from those estimates it is not possible to deduce convergence in relative error. Actually, a complete characterization that answers question \eqref{Q} in the linear case is not known, only partial results have appeared, far from being optimal. In the case of \eqref{FDE} and Weighted \eqref{FDE} (WFDE) with CKN weights, the same result as in this paper has been recently proven by two of the authors, \cite{BS2020}, actually we use some those results in Section \ref{Sec.Grad}, where we exploit the correspondence between radial functions of WFDE and PLE, see \cite{Iagar2008}, to analyze the tail behaviour of (the modulus of) the gradient.

About the \emph{Global Harnack Principle} (GHP) and \emph{convergence in relative error} (CRE). To the best of our knowledge this is the first time that our main results (GHP, CRE, and their characterizations) appear in the literature for the~\eqref{PLE}. In the case of the \eqref{FDE}
convergence in relative error was first proven in~\cite{Vaz2003}, while GHP was first proven in~\cite{BV2006}, under pointwise tail conditions on the data, similar to the $\mathcal{A}_p$ space.  A complete characterization of the GHP and CRE has been recently proven by two of the authors in~\cite{BS2020}, in the case of \eqref{FDE}, also in presence of Caffarelli-Kohn-Nirenberg weights.

In the case of other equations, the validity of CRE and GHP is not always clear. In a quite recent preprint~\cite{Vaz2020b}, a form of GHP for the \emph{fractional p-laplacian} evolution equation is proven to hold for a suitable class of initial data compactly supported. In the case of non-linear equations, we refer to~\cite{Carrillo2019} were the authors prove the CRE for a Newtonian vortex equation. Let us briefly comment about what is known for linear equations. For the classical heat equation, the CRE is known to be false in general, see~\cite{Vazquez2017}. However, to the best of our knowledge, it is an open problem to find sufficient condition for CRE to hold. It is quite surprising that when the diffusion is driven by the \emph{fractional laplacian} $\left(-\Delta\right)^s$, with $s\in (0,1)$, both the GHP and the CRE hold, see~\cite{BonforteSireVaz2017}. However, to the best of our knowledge, no results are available for general linear nonlocal equations.

Our result is important in the study of the \textit{the asymptotic behaviour. }In the good fast diffusion range, $\frac{2N}{N+1}<p<2$, the basic asymptotic result says that  solutions behave for large times as the Barenblatt, in strong $L^q$-topologies. We first provide a complete proof of this fact  in Section \ref{Section:asymptotics} for sake of completeness, by adapting  the so-called 4-Steps Method of \cite{KaminVazquez1988}.
Moreover, rates of convergence to equilibrium (Barenblatt profiles) have been obtained via entropy methods and functional inequalities of Gagliardo-Nirenberg-Sobolev  \cite{delPinoDolbeaultDNLE}, for solutions  with bounded first and second moments. In \cite{ABC} rates of convergence are obtained using linearization and weighted Hardy-Poincar\'e type inequalities, in the spirit of \cite{BBDGV}. However, they require a stronger GHP imposed on $u_0$, namely that it is trapped between two Barenblatt profiles with the same center of mass and without time shifts. This allows them to prove CRE, essential to justify the rest of the proofs. The latter assumption is indeed quite restrictive, and one may wonder whether or not it is optimal. Indeed, such strong pointwise condition on the data is not necessary, we show here that the right hypothesis of \cite{ABC} should be relaxed to $u_0\in \X$, which is enough to have CRE, actually it is equivalent, as we explain in Theorem \ref{ghp.thm}. In other words, we are able to characterize explicitly the basin of attraction of the Barenblatt profiles in the  topology induced by uniform convergence in relative error.

\subsection{Main result 2. Global gradient estimates}
The \eqref{PLE} is a gradient-driven diffusion. Therefore, it is quite natural to investigate the possible decay at infinity (in time and space) of the modulus of the gradient of the solution to the Problem~\eqref{PLEcauchy}. The following theorem provides an optimal answer, when the initial datum is assumed to be in $\X$.
\begin{theorem}[\textbf{Sharp Gradient Estimates}]\label{Thm:gradient.decay}
Let $N\ge 1$ and $\frac{2N}{N+1}<p<2.$ Let $u$ be the solution of Problem~\eqref{PLEcauchy} with $0\le u_0 \in L^1(\RN)$.
Then, there exists a constant $c_1=c_1(N,p)>0$ such that
\begin{equation} \label{time.gradient.decay}
\|\nabla u (t) \|_{L^\infty(\RN)}\le\, c_1\, \frac{\|u_0\|_{L^1(\RN)}^{2\beta}}{t^{(N+1)\beta}} \qquad\mbox{for any }t>0.
\end{equation}
Moreover, if $0\le u_0 \in \X$,  then there exits a constant $c_2=c_2(N, p)>0$ such that
\begin{equation}\label{space.gradient.decay}
|\nabla u(x,t)| \le c_2 \,\frac{   \|u_0\|_{L^1(\RN)}^{2\beta} + \|u_0\|_{\X}^{2\beta}  +t^\frac{2\beta}{2-p} }{\left(1+|x|\right)^{\frac{2}{2-p}}t^{(N+1)\beta}}\qquad\mbox{for any $x\in \RN$ and $t>0$}\,.
\end{equation}
\end{theorem}
\noindent\textbf{Remark. }The above estimates are sharp in the following sense. As for~\eqref{time.gradient.decay}, simple (but lengthy) computations show that $\|\nabla \mathcal{B}(x,t,;M)\|_{L^\infty(\RN)}=c\,t^{-(N+1)\beta}\,M^{2\beta}$, and that the maximum is assumed on the curve $t^\beta=M^{(2-p)\beta}\,h\,|x|$ for $t>0$, where $h=h(p, n)>0$ is a constant. As for~\eqref{space.gradient.decay}, we see that the right-hand-side of that inequality meets the space-time behaviour of the  gradient of the  Barenblatt profile when $t^\beta \le C |x|$, and therefore it is sharp in the class of data $\X$.  Also, it is possible to construct  counterexamples if the condition $u_0\in\X$
dropped, analogous to those constructed in  Section \ref{Section:DataNotX}.

In the case of radial, decreasing initial data we have an even better result.

\begin{theorem}[\textbf{Convergence in relative error for radial derivatives}]\label{Thm:gradient.convergence}
Let $N\ge3$ and $\frac{2N}{N+1}<p<2.$ Let $u$ be the solution of Problem~\eqref{PLEcauchy} with datum $0\le u_0  \in \X\cap C^2(\RN)\setminus\{0\}$ radial and nonincreasing, and let $M=\|u_0\|_{L^1(\RN)}$. If $|\nabla u_0 |^{\frac{p}{2}}\in \mathcal{A}_p$, that is, if there exist $A>0$ and $R_0>0$ such that
\begin{equation}  \label{gradient.decay.hp}
|\partial_r u_0 (r)| \le  A r^{-\frac{2}{2-p}}\,,\quad \mbox{for all }\,\,  r\ge R_0\,,
\end{equation}
then, the following limit holds
\begin{equation}\label{gradient.decay.inq}
\Big\|\frac{\partial_r u(\cdot,t)}{\partial_r \mathcal{B}(\cdot,t; M)}-1\Big\|_{L^\infty(\RN)}\xrightarrow[]{t\to +\infty} 0 \,,
\end{equation}
where $\partial_r u(r,t)$ (resp. $\partial_r \mathcal{B}(r,t; M)$) is the radial derivative of $u(r,t)$ (resp. $\mathcal{B}(r,t; M)$).
\end{theorem}

\begin{remark} \rm
\begin{enumerate}[leftmargin=*]\itemsep0pt
\item Roughly speaking, the above Theorem says that radial data that decay faster (or equal) than the Barenblatt, and satisfy the corresponding tail condition for the radial derivative, produce solutions that converge uniformly in relative error to the Barenblatt with the same mass, in the $C^1$  topology.

\item We also notice that the dimension restriction $N\ge3$ is merely technical and the result extended to the case $N=1, 2$ with minor modifications. Also, the $C^2$ regularity assumption on the initial datum  can be removed by means of tedious but straightforward approximations that we have decided to avoid here.
\item It is important to understand  whether or not the radial monotonicity is necessary. The answer is yes for $N=1$, and we believe that the same happens in higher dimensions, even if we do not have explicit counterexamples.
Let us show the counterexample in one dimension. We begin by recalling the relation between the~\eqref{PLE} and the \eqref{FDE}, which is quite simple when $N=1$, see~\cite{Iagar2008} for higher dimensions. Indeed,  if $u$ is a sufficiently regular solution to the Cauchy problem for~\eqref{PLE}, then its spatial derivative $v=\partial_x u$ solves the corresponding Cauchy problem for the \eqref{FDE} with $m=p-1\in (0,1)$. Hence, we need the initial datum $u_0$ to be regular enough, so that $v_0=\partial_x u_0$ is admissible data for the \eqref{FDE} and we have the correspondence that we need for all times.
We consider now a non-negative initial datum $u_0$ with positive mass $M=\int_{\R}u_0dx>0$ and with zero-mass spatial derivative, namely $\int_{\R}v_0=0$. Such choice is always possible, since, it suffices to take a compactly supported function $v_0$ such that $v_0(-x)=-v_0(x)$ and to define $u_0(x)=\int_{0}^xv_0(y)dy$.  By maximum principle and conservation of mass we have that $u(t)\ge0$ and has positive mass. Its derivative $v(t)$  preserves the zero mass of the initial datum, namely $\int_{\R}v(t)=0$, hence (being non-constant) it must change sing. We conclude that there exists a point $x(t)$ for which $v(x(t), t)=0$ and so
\[
\left|\frac{v(x(t),t)}{ \partial_x\mathcal{B}(x(t),t; M)}-1\right|=1\,,\quad\mbox{for any}\,\,t>0\,,
\]
which concludes the counterexample.

\end{enumerate}
\end{remark}

\subsection{Structure of the paper}\label{Section:Literature}

Section \ref{Section:Preliminaries} gathers information on the existing results about the \eqref{PLE} relevant to our purposes, together with some preliminaries needed throughout the paper. The rest of the paper is essentially devoted to the proof our two main results.
In Section \ref{sec:upper.lower}, we prove quantitative upper and lower bounds for the solution to Problem \eqref{PLEcauchy}, which fairly combine in the GHP. This shows how  i) implies ii) in Theorem \ref{ghp.thm}.

In Section \ref{Section:asymptotics}, we  show that $\X$ data leads to the convergence in relative error for the solution $u$ to Problem \eqref{PLEcauchy} to the unique Barenblatt solution $\mathcal{B}$ with same mass, that is  i) implies iii) in Theorem \ref{ghp.thm}. This relies on a basic asymptotic result,  the $L^q$ convergence of $u$ to $\mathcal{B}$, for which we provide a proof adapted from \cite{KaminVazquez1988} for sake of completeness.
In Section \ref{Section:Counterexample}, we focus on the space $\X$. Firstly, we conclude  the proof of Theorem  \ref{ghp.thm}.  We also prove the equivalent characterization \eqref{equiv.cond.Xp.0000} of the space $\X$.
Then, we show that space $\X$ is indeed  larger than the space $\mathcal{A}_p$ of functions satisfying the pointwise decay condition \eqref{tail:space},  by constructing a counterexample of a function $g \in \X \setminus \mathcal{A}_p$. In the last part of this section, we focus on what can happen for initial data in $L^1_+ (\RN) \setminus \X$: we construct counterexamples to the GHP, and we show a Generalized GHP.
In Section \ref{Sec.Grad}, we provide the proofs of Theorems \ref{Thm:gradient.decay} and \ref{Thm:gradient.convergence} on  global gradient decay estimates.
Section \ref{sec.conservation.mass} contains the proof of mass conservation, that we have decided to include for the sake of completeness. We conclude with an Appendix
where we gather some useful tools used throughout the paper.

\section{Preliminaries}\label{Section:Preliminaries}

Equations as~\eqref{PLE} have been investigated since the 70's due to their intrinsic mathematical difficulties and the wide range of applications, see~\cite{Kalashnikov1987, Vaz17} and also the monographs~\cite{Vazquez2007,JLVSmoothing,Zhuoqun2001}. In the range $p>2$~\eqref{PLE} has been widely investigated (see~\cite{DiBenedetto} and references therein), while in the so-called \emph{good fast diffusion range} $\frac{2N}{N+1}<p<2$, less results are present in the literature. We refer to the monographs~\cite{DGV-Book}, \cite[Part III]{Vazquez2007} and references therein for further information.

In what follows we recall some important properties of the \eqref{PLE} equation as well as special classes of solutions.
We first introduce the concept of weak solution to Problem \eqref{PLEcauchy} that we will use in this paper.
\begin{defn} We say that $u$ is a weak solution  to Problem \eqref{PLEcauchy}  with initial data $u_0 \in L^1_{\text{loc}}(\RN)$ if $u\in
L^1((0,T):W^{1,p}_{\text{loc}}(\RN))$ and
\begin{align}\label{weak.formulation}
\int_{\RN}u(x, s)\phi(x,s)dx &= \int_{\RN}u(x,t)\phi(x,t)dx \\  \nonumber
&+ \int_s^t\int_{\RN}\big(-u(x,\tau)\phi_\tau(x,\tau) + |\nabla u(x,\tau)|^{p-2}\nabla u(x,\tau) \cdot \nabla \phi(x,\tau) \big)dxd\tau\,,
\end{align}
for all $0<s<t<+\infty$ and for all functions $\phi \in C_c^\infty(\RN \times (0,T)).$ The initial data is taken in the sense:
 $$\lim_{t\to 0} \int_{\RN}u(x,t)\varphi(x) dx =\int_{\RN} u_0(x) \varphi(x) dx, \quad \forall \varphi \in C_c^{\infty}(\RN). $$
\end{defn}
This definition is equivalent to the one used in \cite[III.2.3]{DiBenedetto1990}, where existence and uniqueness (and much more) are settled for the problem under consideration. In what follows, we will work mostly with smooth compactly supported test functions, as specified in the above definition. However,  a straightforward density argument allows to extend the equality \eqref{weak.formulation}  to test functions with a suitable decay at infinity, not necessarily compactly supported. This fact will be used in \eqref{herrero.pierre.infty}.

In the fast diffusion range, $p\in (1,2)$, we refer to~\cite{DiBenedetto1990}, where a quite complete \textit{existence and uniqueness }theory for the Cauchy problem has been done, together with a number of useful basic estimates, for locally integrable data. Later interesting development are contained in ~\cite{Zhao1995}. More details will be given in what follows.

\subsection{Reminder about existing results}\label{ssec.Basic.Theory}

Let us briefly recall the common basic theory and point out the differences that occurs in the distinct ranges and values of the parameter $p\ge 1$.  The literature on the p-Laplacian is so vast that it is hopeless to make it complete here, our aim is to point out the result and sources used in the paper, apologizing for unfortunate and involuntary omissions.

\noindent$\bullet~$\textit{Existence and uniqueness, }for different concepts of solutions, weak, mild (semigroup), strong, is nowadays considered quite standard:
The equation is known to generate a contraction semigroup in all $L^q(\RN)$ spaces, with $q\in [1,\infty]$, \cite{BCPbook, Br-Book, CL71, Komura, Ambrosio, AGSbook}, whenever $p>1$.

On one hand, the case $p=1$ corresponds to the 1-Laplacian or Total Variation Flow, very useful in image processing. However, when $p=1$, the techniques and results are quite different from the case $p>1$, see \cite{AndreuCasellesMazon-Book, BF2012}. On the other hand, the case $p=2$ corresponds to the classical Heat Equation, for which the Gaussian representation formula is a fundamental tool in developing a complete theory of existence and uniqueness, even for growing initial measures, Widder Theory \cite{WIDDER}.  When $p\ne 2$ no representation formula holds, different techniques have to be used, such as nonlinear semigroup theory, developed in the 70s-80s by prominent mathematicians, among them Brezis, Crandall, B\'enilan, Komura, V\'eron, etc. Generation theorems, like the celebrated Crandall-Liggett Theorem \cite{CL71}, that can be considered a nonlinear extension of the Hille-Yosida or Lumer-Phillips Theorems, and provide existence and uniqueness in Banach (or even metric) spaces, has been proven in different setups, see \cite{BCPbook, Br-Book,  Komura,Veron} and also \cite{Ambrosio, AGSbook}. The underlying idea is that the PLE-flow can be seen as the $L^2$ gradient flow of the $p$-energy functional $\tfrac{1}{p}\int |\nabla u|^p dx$, analogously to what happens to the HE.

\noindent$\bullet~$ \textit{Time monotonicity estimates. }  The solution $u$ to Problem \eqref{PLEcauchy} satisfies the celebrated \textsl{Benilan-Crandall estimate}, see for instance \cite{Benilan1981, EstVaz}, which in differential form (and in the sense of distributions) read
\begin{equation*}
u_t (\cdot,t) \le \frac{u(\cdot,t)}{(2-p)t} \quad \text{ for a.e. $t>0$.}
\end{equation*}
This is nothing but a weak formulation of the monotonicity of the map: for a.e. $x\in\RN$ we have that
\begin{equation}\label{BC:decreasing}
 t\to t^{-\frac{1}{2-p} }\, u(x,t) \quad \text{ is a non-increasing function for a.e. $t>0$.}
\end{equation}
As a consequence of the above monotonicity estimates, see  \cite{Benilan1981,DiBenedetto1990}, we have
\begin{equation}\label{BenilanCrandall}
\|u_t\|_{L^1(\RN)} \le \frac{2(2-p)}{t}\|u_0\|_{L^1(\RN)} \quad \text{ for a.e. $t>0$.}
\end{equation}

\noindent$\bullet~$\textit{Finite VS infinite speed of propagation. }As we have already seen in the introduction, the Barenblatt solution \eqref{Berenblatt.000} clearly shows, when $p>2$ its support remains compact for all times, that is finite speed of propagation \cite{HerreroVaz}, while when $p\ge 2$ the support immediately spread on the whole space (see \eqref{BarenblattMass} below), which is infinite speed of propagation. We will focus on the case $p\in (1,2)$, often called ``the fast diffusion -or singular- range'', where the speed of propagation is infinite, as it happens for the HE. Existence and uniqueness of weak solutions for a quite large class of $L^1_{\rm loc}$ initial data was proven by DiBenedetto and Herrero \cite{DiBenedettoHerrero}. The case $p>2$, often called ``the slow diffusion -or degenerate- range'', differs from the singular case and will not be considered here, see \cite{DiBenedetto, DGV-Book, DGV-Acta, KaminVazquez1988}.

\noindent$\bullet~$ \textit{More about the Barenblatt or Fundamental solution to the \eqref{PLE} equation. }When $p\in (p_c,1)$, $p_c=\tfrac{2N}{N+1}$, the Bareblatt solution has an explicit formula given below \eqref{BarenblattMass}. Let us now recall some useful facts that will be systematically used in the rest of the paper, and explain the derivation of such formulae, to setup notations and to make it easier to be checked by the interested reader. Let us consider the equation
\begin{equation}\label{PLE_baren}
\mathcal{B}_t(x,t)=\Delta_p \mathcal{B}(x,t), \quad x\in \RN,\, t>0,
\end{equation}
with initial data the Dirac delta with mass $M$:
\begin{equation}\label{DiracM}
\lim_{t \rightarrow 0}\, \mathcal{B}(x,t)=M \delta_0(x).
\end{equation}
In the range $p>\frac{2N}{N+1}$, the solution of problem \eqref{PLE_baren}-\eqref{DiracM} exist (cf.~\cite{JLVSmoothing}), we denote it by $\mathcal{B}(x,t;M)$  and we will call it the \emph{Barenblatt} solution. We remark that $\mathcal{B}(x, t; M)$ has mass $M$, i.e. $\int_{\RN}\mathcal{B}(x, t; M)dx=M$ . In the fast diffusion range $\frac{2N}{N+1}<p<2$, the Barenblatt solutions have the form
$$
\mathcal{B}(x,t;M)=  t^{-\alpha} F(t^{-\beta} x)= t^{-\alpha} \left( b + b_2\left|xt^{-\beta}\right|^{\frac{p}{p-1}}\right)^{-{\frac{p-1}{2-p}}}\,,
$$
where
\begin{align}\label{parameters}
\alpha=&\frac{1}{p-2+(p/N)}, \quad \beta=\frac{\alpha}{N}=\frac{1}{N(p-2)+p}, \quad b_2=\frac{2-p}{p} \left (\frac{\alpha}{N} \right)^{\frac{1}{p-1}}\,,
\end{align}
while $b>0$ is free and it can be uniquely determined in terms of the initial mass $M$. By self-similarity we can express $\mathcal{B}(x,t;M)$ in terms of $\mathcal{B}(x,t;1)$
\begin{equation}\label{MassRelation}
\mathcal{B}(x,t;M)=M \mathcal{B}(x, t\, M^{p-2};1).
\end{equation}
If we denote by $\mathcal{B}(x,t;1)=t^{-\alpha} F_1(|x|t^{-\beta})$, then
\begin{equation}\label{self.similarity.relation}
\mathcal{B}(x,t;M)=M \left(M^{p-2}\,t\right)^{-\alpha}\,F_1\left(|x|\, \left(t\, M^{p-2}\right)^{-\beta}\right)=t^{-N\beta}M^{p\beta}F_1\left(|x|\, \left(t\,
M^{p-2}\right)^{-\beta}\right)
\end{equation}
Let $b_1$ be the parameter corresponding to the Barenblatt solution with mass $1$. That is, $b_1$ is a positive constant such that
 $$
 \int_{\RN}\left( b_1 + b_2 |x|^{\frac{p}{p-1}} \right)^{-{\frac{p-1}{2-p}}}dx =1.
 $$
 Then, by using formula \eqref{MassRelation}, we obtain the expression of $\mathcal{B}(x,t;M)$ depending on the mass:
\begin{equation*}
\mathcal{B}(x,t;M)= t^{\frac{1}{2-p}} \left[ b_1 \frac{t^{\frac{\beta p}{p-1}}}{M^{(2-p)\frac{\beta p}{p-1}}} + b_2\left|x\right|^{\frac{p}{p-1}}\right]^{-{\frac{p-1}{2-p}}}.
\end{equation*}
Lastly, let us observe that the expression of $F_1(|y|)$ is given by
\begin{equation}\label{F1}
F_1(|y|)=\left(b_1+b_2|y|^\frac{p}{p-1}\right)^\frac{1-p}{2-p}\,.
\end{equation}

\noindent$\bullet~$ \textit{Singular Barenblatt solutions. }  The singular Barenblatt solution $\mathcal{U}(x,t)$ is defined as the limit when $M\to \infty$ of the Barenblatt solution $\mathcal{B}(x,t;M)$. We can write the singular Barenblatt solution starting at time $S$ as
\begin{equation}\label{SingBB1}
\mathcal{U}(x,t;S)=b_2^{-{\frac{p-1}{2-p}}}  (t+S)^{\frac{1}{2-p}} |x|^{-{\frac{p}{2-p}}}.
\end{equation}
It is known that $\mathcal{U}(x,t;S)$ is a supersolution of the \eqref{PLE} equation in the domain $\{|x|>0, t>0\}$:
$$\mathcal{U}_ t\ge \Delta_p \mathcal{U}, \quad \forall |x|>0, t>0.$$

\noindent$\bullet~$\textit{Conservation of mass, Extinction Time, and the good fast diffusion range. }When $p>p_c$, integrable data $u_0\in L^1(\RN)$ produce solutions $u(\cdot,t)$ that preserve mass, namely
\[
\int_{\RN}u(x,t) dx=\int_{\RN}u_0(x) dx\qquad\mbox{for all }t>0.
\]
This is well known for the Cauchy problem for the classical HE, and in the degenerate case has been proven in \cite{KaminVazquez1988}.  We provide in Section \ref{sec.conservation.mass} a proof for the case $p_c<p<2$, for the sake of completeness. Another proof can be found in \cite{FDV}, where more general p-Laplacian type operators are considered. Recently, a proof of mass conservation in the fractional case has been given in \cite{Vaz2020,Vaz2020b}, which allows to recover the same result for the local case, using a delicate limiting process. When $p\in (1,p_c)$ the conservation of mass fails and solution extinguish in finite time, see for instance \cite{DiBenedettoHerrero,BIV}. We refrain from giving further details, since we will not treat this case here.

\noindent$\bullet~$\textit{\textit{Mass rescaling. }}In this paper we will systematically use the following scaling argument. If $v(x,t)$ is a solution to the \eqref{PLEcauchy} with initial datum $v_0$ such that  $\int_{\RN}v_0(x) dx=M$, it is easy to see that
\begin{equation}\label{transf}
u(x,t)= \frac{1}{M} v(x,t M^{2-p})
\end{equation}
is a solution to the \eqref{PLE} with mass 1, that is, $\int_{\RN}u_0(x) dx=1$ and $u(\cdot,0)  = \frac{1}{M}v_0$. Therefore, it is sufficient to work with solutions with unitary mass, then the rescaling \eqref{transf} allows to recover the case with $M\ne 1$.

\noindent$\bullet~$\textit{Local Regularity: Boundedness, Positivity, H\"older continuity and higher regularity estimates.}
In the good fast diffusion range $p\in(p_c,2)$, these issues were addressed since the 80s, see~\cite{DiBenedetto1990,DiBenedettoFriedman,DiBenedetto1984,DGV2007} and the monographs~\cite{DiBenedetto, DGV-Book}.
Precise form of the $L^\infty$ regularity estimates are given in the proofs, where are used, together with references. The basic smoothing effect, is the global $L^1-L^\infty$ estimate: $\|u(t_0)\|_{L^\infty(\RN)} \lesssim \|u_0\|_{L^1(\RN)}^{p\beta} t_0^{-N\beta}$, see \cite{BIV,DiBenedetto,DiBenedetto1990,DGV-Book}. As for $L^\infty-C^{\alpha}$ estimates, see \cite{DiBenedetto,DiBenedetto1990,DiBenedettoFriedman,DGV2007,DGV-Book}, we refer to Appendix \ref{Appendix.Regularity}, where we show a $L^1-C^\nu$ global estimate as a consequence of the above mentioned results, see Lemma \ref{Lemma:reg} where we show that $\lfloor u(t) \rfloor_{C^\nu(\mathbb{R}^N)}\lesssim   \|u_0\|_{L^1(\mathbb{R}^N)}^{p\beta}t^{-\eta}\,,$ for some $\eta>0$.  We remark that when $p<p_c$, in the so-called very fast diffusion range, $L^1$ data may not produce bounded solutions, however in \cite{BrFr1983,BIV,DGV-Book} local upper bounds are provided for all $p\in (1,2)$, with some additional integrability on the initial datum when $p<p_c$, but we shall not consider this latter case in this paper.

Once boundedness is settled, the next question is positivity. Precise local lower bounds are proven in \cite{BIV,DGV-Book}, and fairly combined with upper bounds, they provide Harnack inequalities, which for the nonlinear equations under investigation, take an intrinsic form, meaning that the size of the parabolic cylinder where the estimate holds depend on the solution itself. In the very fast diffusion regime, $1<p\le p_c$,  new form  of Harnack inequalities have been first shown in 2010, see \cite{BIV} and then generalized in \cite{FornaroSosioVespri, FornaroSosioVespriHarnack}, and \cite{DGV-Book}.

As far as higher regularity is concerned, gradient estimates are next. To the best of our knowledge only local estimates are known. Local $L^1-L^\infty$ smoothing effects have been proven in~\cite[Lemma 2.5, pag.621]{Zhao1995}, by means of a De Giorgi type iteration, see also Section \ref{Sec.Grad}. Later, a breakthrough in regularity was obtained: $C^{1,\alpha}$ estimates for measure data problem, that is $C^\alpha$ estimates for the modulus of the gradient, obtained through novel and surprising nonlinear potential estimates, see \cite{DM1,DM2,KM1,KM2,KM3,KM4, BDKS}.

\noindent$\bullet~$\textit{Related equations and Asymptotic behaviour. }The theory of~\eqref{PLE} is closely related to the one of \emph{Porous Medium}/\emph{Fast Diffusion }(PME/\eqref{FDE}) equation $u_t=\Delta u^m$ where $0<m<\infty$, both from qualitative and quantitative points of view.  For radial solutions, the similarity is even stronger since there is a complete equivalence, as proved in~\cite{Iagar2008}. Indeed, in 1D, this equivalence is pretty clear: solutions to PME/FDE equations are the spatial derivatives of solutions to~\eqref{PLE}. We refer to Section 6 especially Theorem \ref{radial.transformation.thm} for more details.

The \eqref{PLE} is a particular case of the \emph{Doubly Nonlinear Diffusion Equation}
\[
u_{t}(x,t) = \Delta_p u^m(x,t), \quad t \in (0,+\infty),\, x \in \mathbb{R}^N. \tag{DNLE}\label{DNLE}
\]
the key parameter in this case is $m(p-1)$, where $m>0, p>1$. When $m(p-1)>1$ slow diffusion occurs, with finite speed of propagation, analogously to the PLE case when $p>2$, while when $m(p-1)<1$ we are in the fast diffusion regime. The \textit{good fast diffusion case }corresponds to the range of parameters $1-(p/N)<m(p-1)<1$. Indeed, when $m=1$ we recover the range $\frac{2N}{N+1}<p<2$ considered in this paper.

For basic results on existence, boundedness and regularity results for equation~\eqref{DNLE} we refer to \cite{Be88, BDMS, BG06, CG1,CG2, DiBenedetto,DGV-Book, EbmeyerUrbano, KK07,KuusiUrbanoDNLE,Is96, Ivanov, MV94, PorzioVespri, SV94}. As far as asymptotic behaviour is concerned, nonnegative integrable solutions behave like Barenblatt for large times, in perfect analogy to the PLE case, see also \cite{JLVSmoothing}. We have already mentioned the two closest results, \cite{ABC,delPinoDolbeaultDNLE}, to which we would like to add \cite{AguehConvexityRegime,AguehMongeKant,AguehRates}.

\medskip

Once the local theory of nonnegative solutions has reached a satisfactory level, the next task is to investigate the global theory, that is provide the sharpest possible answer to question \eqref{Q}.  Similar questions will be addressed in Section \ref{Sec.Grad}, for the modulus of the gradient of the solutions.\vspace{-4mm}

\section{Global upper and lower pointwise estimates}\label{sec:upper.lower}\vspace{-2mm}
For data in the $\X$ space, the solution of the Cauchy problem \eqref{PLEcauchy} can be bounded from above and below by a suitable Barenblatt profiles. The upper bound is given in the following theorem:
\begin{theorem}\label{ThmUpperBound2}
Let $N\ge 1$ and $\frac{2N}{N+1}<p<2.$ Let $u$ be a weak solution to Problem \eqref{PLEcauchy} corresponding to an initial datum  $0\le u_0\in \X$. Then for any $t_0 >0$ there exist constants $\tau_2, \, M_2$ such that  the following upper bound
holds\vspace{-2mm}
\begin{equation}\label{upperboundXdata}
u(x,t) \le \mathcal{B}(x,t+\tau_2;M_2),\qquad \text{for all} \quad x\in \RN,\, t\in (t_0,\infty).\vspace{-2mm}
\end{equation}
\end{theorem}
The lower bound reads:
\begin{theorem}\label{lower.estimate.theorem}
Let $N\ge 1$ and $\frac{2N}{N+1}<p<2.$ Let $u$ be a weak solution to Problem \eqref{PLEcauchy} corresponding to an initial datum  $ u_0\in L^1(\RN)$, let  $R_0>0$ be such that $ \|u_0\|_{L^1(B_{R_0}(0))} > 0$, and let $t_0>0$. Then there exist $\underline{\tau} > 0$  and
$\underline{M}>0$ such that\vspace{-2mm}
\begin{equation}\label{lower.estimate.inequality}
u(x,t) \geq \mathcal{B}(x, t-\underline{\tau}; \underline{M})\,,\quad\mbox{for all}\,\, x\in \RN\,\, \mbox{and }\,\,t\ge t_0\,,\vspace{-2mm}
\end{equation}
where, being $t_c$  as in~\eqref{critical.time}, the parameters $\underline{\tau}$ and $\underline{M}$ are given by:
\begin{enumerate}
\item[i)] If $t_0\ge t_c$:
\begin{equation}\label{choice.case.i}
\underline{\tau}=  a t_c  \quad\mbox{and}\quad \underline{M}=b \, \|u_0\|_{L^1(B_{R_0}(0))}.
\end{equation}
\item[ii)] If $0<t_0\le t_c$:
\begin{equation*}
\underline{\tau}=  a t_0  \quad\mbox{and}\quad \underline{M}=b \, \|u_0\|_{L^1(B_{R_0}(0))} \left(\frac{t_c}{t_0}\right)^{-\frac{1}{2-p}}.
\end{equation*}
\end{enumerate}
In both cases, the constants $a,\, b>0$ depend only on $N$ and $p$ and are given in~\eqref{choice.a.b}.
\end{theorem}

We provide the proofs of these two theorems.

Firstly, we will prove an auxiliary result concerning the upper bound in terms of a Barenblatt profile  for solutions $u$ to Problem \eqref{PLEcauchy} corresponding to suitable
pointwise decaying data.
\begin{proposition}\label{ThmUpperBound}
Let $N \ge 1$ and $\frac{2N}{N+1}<p<2$. Let $u_0 \in L^1(\RN)$, $u_0\ge 0$, $\int_{\RN}u_0(x) dx =1$ and
\begin{equation}\label{decayu0}
u_0(x) \le A |x|^{-\frac{p}{2-p}} \quad \text{for }  x\in \RN\setminus\{0\}.
\end{equation}
Let $u$ be the weak solution to Problem \eqref{PLEcauchy} with initial data $u_0$.
Then, for any $t_0 >0$, there exist constants $\tau_2, \, M_2$ such that for any $x \in \RN$, $t\in (t_0 , \infty)$ we have the following upper bound
\begin{equation}\label{upperbound}
u(x,t) \le \mathcal{B}(x,t+\tau_2;M_2),
\end{equation}
where $\tau_2 =\tau_2(t_0,A)$, $M_2=M_2(t_0,A,\tau_2)$.
\end{proposition}
\begin{remark}  The proof of  Proposition \ref{ThmUpperBound}  can be extended without any difficulties to data satisfying the more general $u_0\in\mathcal{A}_p$. This has been done in \cite{BV2006} for the fast diffusion equation $u_t=\Delta u^m$, $m<1$. We mention that our strategy is a simplified version of the one given \cite{BV2006}.
\end{remark}

\begin{proof}[Proof of Proposition \ref{ThmUpperBound}.]
Let us fix a value $t_0>0.$ It is sufficient to prove estimate \eqref{upperbound} for time $t=t_0$. Then, for larger times $t>t_0$, the result follows by comparison principle. We prove there exists a suitable
choice of parameters $M_2$ and $\tau_2$ such that
\begin{equation}\label{estim:t0}
u(x,t_0) \le \mathcal{B}(x,t_0+\tau_2;M_2), \quad \forall x \in \RN.
\end{equation}
\begin{figure}[h!]
\begin{center}
\includegraphics[scale=0.5]{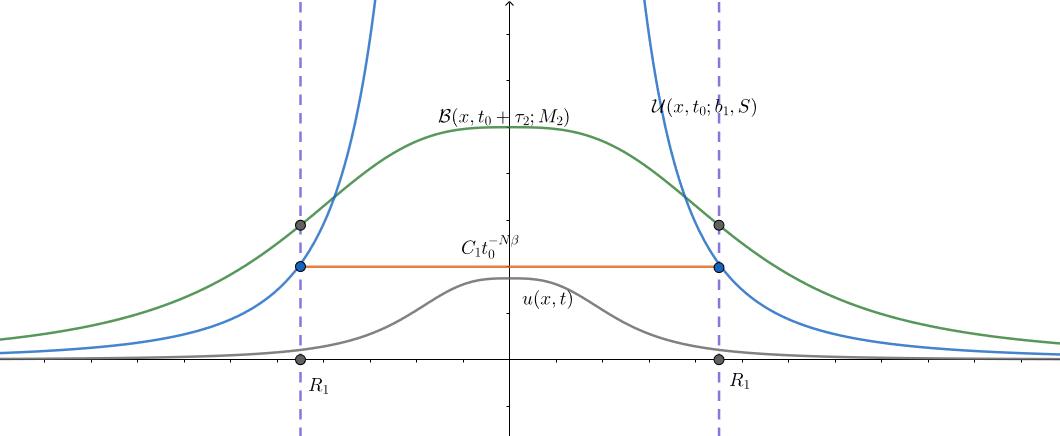} \caption{\label{figura1}Construction of the barriers for the proof of the upper bound}
\end{center}
\end{figure}
The strategy is as follows: in view of the decay of the data \eqref{decayu0}, firstly, we determine sufficient conditions  for the solution $u(x,t_0)$ to be bounded from above by a
singular Barenblatt solution $\mathcal{U}$ for $|x|>0$.
This is an upper barrier  which meets the upper bound given by the smoothing effect \eqref{smoothingEffect} at some point $|x|=R_1$.
Then we find the Barenblatt solution $\mathcal{B}(x,t_0+\tau_2;M_2)$ to be above the barrier $\mathcal{U}$ for all $|x| \ge R_1$, and therefore it will be above $u(x,t_0)$. Inside the ball
$\{ |x|<R_1 \}$ the comparison \eqref{estim:t0} follows by the monotonicity of $\mathcal{B}(x,t_0+\tau_2;M_2)$ in $|x|$. See Figure \ref{figura1}.

\noindent Step I.\emph{ Upper barrier outside a ball.} We consider the singular Barenblatt solution starting at time $S$ as it was previously introduced in \eqref{SingBB1}
$$
\mathcal{U}(x,t;S)=b_2^{-{\frac{p-1}{2-p}}}  (t+S)^{\frac{1}{2-p}} |x|^{-{\frac{p}{2-p}}}.
$$
We continue by proving that, under certain conditions, we can compare the initial data $u_0$ with  $\mathcal{U}(x,0;S)$ for an appropriate $S$. Let us choose:
\begin{equation*}
S \ge b_2^{p-1}A^{2-p}.
\end{equation*}
Then
 $$
 u_0(x) \le A |x|^{-\frac{p}{2-p}} \le \mathcal{U}(x,0;S),\quad \forall x\in \RN, x\neq 0.
 $$
It is known that $\mathcal{U}(x,t;S)$ is a supersolution of the equation in the domain $\{|x|>0, t>0\}$. Since $\mathcal{U}(0,t;S)=+\infty>  u(0,t)$ for all $t>0$, we conclude by
the maximum principle that
$$u(x,t)\le \mathcal{U}(x,t;S)  ,\quad \forall x\in \RN, t>0.$$

Step II. \emph{ Upper estimates in the whole space. } We determine the point $R_1$ where the upper barrier $\mathcal{U}(x,t;B_1,S)$ meets the one given by using the following smoothing effect at time $t=t_0$:
\begin{equation}\label{smoothingEffect}
\|u(t_0)\|_{L^\infty(\RN)} \le C_1 \|u_0\|_{L^1(\RN)}^{p\beta} t_0^{-N\beta}
\end{equation}
The above estimate has been proven in \cite[Thm.11.2]{JLVSmoothing} with sharp constant $C_1= C_1(p,N)>0$. Another proof, without sharp  (yet explicit) constant, follows by letting $R_0\to \infty$ in the local smoothing effect \eqref{local.smoothing.effect}.
Recalling that we have fixed $\|u_0\|_{L^1(\RN)}=1$, we let
\[
C_1t_0^{-N\beta}= \mathcal{U}(R_1,t_0;S).
\]
This way we obtain a first upper bound for the solution $u(x,t)$ in the whole space:
\[
u(x,t) \le \left\{\begin{array}{lr}
        C_1t_0^{-N\beta}, & \text{for } |x|\leq R_1\\[2mm]
        \mathcal{U}(x,t_0;S), & \text{for }|x|\ge R_1.
        \end{array} \right.
\]
Equivalently,
\begin{equation}\label{R1def}
b_2R_1^{\frac{p}{p-1}} =  \frac{t_0^{{\frac{N\beta(2-p)}{p-1}}} (t_0+S)^{\frac{1}{p-1}} }{C_1^{{\frac{2-p}{p-1}}} } .
\end{equation}
We take $R_1$ defined by this formula.

Step III. \emph{Finding the Bareblatt solution}.
We search for $\mathcal{B}(x,t+\tau_2;M_2) $ such that at time $t=t_0$ we have that
\begin{equation}\label{Comp:BB:U}
\mathcal{B}(x,t_0+\tau_2;M_2)  \ge \mathcal{U}(x,t_0;S)  , \quad \forall |x| \ge R_1.
\end{equation}
Observe that by the monotonicity in $|x|$ one has for $|x| \le R_1$ the following comparison:
$$ u(x,t) \le C_1 t_0^{-N \beta} = \mathcal{U}(R_1,t_0;S) \le \mathcal{B}(R_1,t+\tau_2;M_2) \le \mathcal{B}(x,t+\tau_2;M_2).$$

Now, we investigate condition \eqref{Comp:BB:U}. This is equivalent to
$$b_2|x|^{\frac{p}{p-1}} \left[ ( t_0+S)^{-\frac{1}{p-1}} - (t_0+\tau_2)^{-\frac{1}{p-1}}\right] \ge  b_1 \frac{(t_0+\tau_2)^{\frac{\beta p -1}{p-1}}}{M_2^{(2-p)\frac{\beta
p}{p-1}}},
\quad \forall |x|\ge R_1.
$$
This inequality is of the form $ |x|^{\frac{p}{p-1}} \ge C$. The left hand side term is an increasing function of $|x|$. It is sufficient to take
 $R_1^{\frac{p}{p-1}}\ge C $ and this implies $ |x|^{\frac{p}{p-1}}\ge C$ for all $ |x| \ge R_1.$
Hence, we impose the condition
\begin{equation}\label{cond1}
b_2R_1^{\frac{p}{p-1}} \left[ ( t_0+S)^{-\frac{1}{p-1}} - (t_0+\tau_2)^{-\frac{1}{p-1}}\right] \ge  b_1 \frac{(t_0+\tau_2)^{\frac{\beta p -1}{p-1}}}{M_2^{(2-p)\frac{\beta p}{p-1}}}
.
\end{equation}
We use now the definition of $R_1$ given by \eqref{R1def}. The last condition is equivalent to
\begin{align*}
 \frac{t_0^{\frac{N\beta( 2-p)}{p-1}}} {C_1  ^{{\frac{2-p}{p-1}}}}  \left[(t_0+\tau_2)^{\frac{1}{p-1}}- (t_0+S )^{\frac{1}{p-1}}\right] \ge b_1
\frac{(t_0+\tau_2)^{\frac{\beta p }{p-1}}}{M_2^{(2-p)\frac{\beta p}{p-1}}} + b_2 \left( \frac{R_0}{S^{\beta}} \right)^{\frac{p}{p-1}}  (t_0+S)^{\frac{\beta p }{p-1}} .
\end{align*}
Recall that $N\beta( 2-p) = \beta p -1$. Thus the previous condition is equivalent to
$$  \frac{t_0^{\frac{\beta p -1}{p-1}}} {C_1^{{\frac{2-p}{p-1}}}}  (t_0+\tau_2)^{\frac{1}{p-1}}\ge b_1 \frac{(t_0+\tau_2)^{\frac{\beta p
}{p-1}}}{M_2^{(2-p)\frac{\beta p}{p-1}}} + b_2 \left( \frac{R_0}{S^{\beta}} \right)^{\frac{p}{p-1}}  (t_0+S)^{\frac{\beta p }{p-1}}  +  \frac{t_0^{\frac{\beta p -1}{p-1}}}
{C_1^{{\frac{2-p}{p-1}}}}(t_0+S )^{\frac{1}{p-1}}.$$

Therefore, we take $\tau_2$ such that
\begin{equation}\label{cond:tau2}
\frac{1}{2} \frac{t_0^{\frac{\beta p -1}{p-1}}} {C_1^{{\frac{2-p}{p-1}}}}  (t_0+\tau_2)^{\frac{1}{p-1}}\ge b_2 \left( \frac{R_0}{S^{\beta}} \right)^{\frac{p}{p-1}}
(t_0+S)^{\frac{\beta p }{p-1}}  +  \frac{t_0^{\frac{\beta p -1}{p-1}}} {C_1^{{\frac{2-p}{p-1}}}}(t_0+S )^{\frac{1}{p-1}}.
\end{equation}
This is equivalent to
\begin{equation*}
(t_0+\tau_2)^{\frac{1}{p-1}} \ge   2 b_2 C_1^{-{\frac{2-p}{p-1}}}
R_0^{\frac{p}{p-1}}  \left(\frac{1}{S}+\frac{1}{t_0}\right)^{\frac{\beta p }{p-1}} t_0^{\frac{1}{p-1}} + 2 (t_0+S )^{\frac{1}{p-1}}.
\end{equation*}
Let $$K:=2 b_2 C_1^{-{\frac{2-p}{p-1}}}
R_0^{\frac{p}{p-1}}.$$
 We take:
\begin{equation}\label{def:tau2}
\tau_2 = \max\left( \left[ K \left(\frac{1}{S}+\frac{1}{t_0}\right)^{\frac{\beta p }{p-1}} t_0^{\frac{1}{p-1}} + 2 (t_0+S )^{\frac{1}{p-1}} \right]^{p-1} -t_0, t_0/2\right).
\end{equation}
With the value of $\tau_2$ given by this formula, we take $M_2$ such that
\begin{equation}\label{cond:M2}
\frac{1}{2}\frac{t_0^{\frac{\beta p -1}{p-1}}} {C_1^{{\frac{2-p}{p-1}}}}  (t_0+\tau_2)^{\frac{1}{p-1}}\ge b_1 \frac{(t_0+\tau_2)^{\frac{\beta p }{p-1}}}{M_2^{(2-p)\frac{\beta
p}{p-1}}}.
\end{equation}
This is equivalent to
\begin{equation}\label{def:M2}
  M_2^{(2-p)\frac{\beta p}{p-1}} \ge 2 b_1 \, C_1^{-\frac{2-p}{p-1}} \left (\frac{t_0+\tau_2}{t_0} \right)  ^{\frac{\beta p -1}{p-1}} .
\end{equation}
Thus, by defining $M_2$ and $\tau_2$ by \eqref{def:M2} and \eqref{def:tau2}, conditions \eqref{cond:M2} and \eqref{cond:tau2} are satisfied, and therefore \eqref{cond1} is
satisfied.
\end{proof}
\begin{proof}[Proof of Theorem \ref{ThmUpperBound2}]
Firstly, we prove that, given data $u_0\in \X$, the solution $u(x,t)$ at any positive time $t>0$ will satisfy an upper bound of the form \eqref{decayu0}. To this end we will use a local $L^1-L^\infty$ smoothing effect, which has been proven in \cite[Thm. III.6.2]{DiBenedetto1990}, \cite[Thm 3.1]{BIV} and reads: for every $x_0\in \RN$ and all $R_0>0$ we have
\begin{equation}\label{local.smoothing.effect}
 \sup_{B_{R_0}(x_0)} u(x, t) \le  \kappa_1 \frac{\left(\int_{B_{2R_0}(x_0)}u_0\,dx\right)^{p\beta}}{t^{N\beta}} + \kappa_2 \,\left(\frac{t}{{R_0}^{p}}\right)^{\frac{1}{2-p}}.
\end{equation}
We will use it with the admissible choice $R_0=\frac{|x_0|}{4}$, so that $B_{2R_0}(x_0)\subset \RN\setminus B_{R_0}(0)$. Next, using the definition of the space $\X$ and  recalling that $u_0 \in \X$, we get
$$
R_0 ^{\frac{p}{2-p}-N}\int_{B_{2R_0}(x_0)}u_0(y) dy \le R_0 ^{\frac{p}{2-p}-N} \int_{\RN\setminus B_{R_0}(0)}u_0(y) dy  =: C(u_0) <+\infty.
$$
Thus
$$\int_{B_{2R_0}(x_0)}u_0(y) dy \le R_0 ^{N-\frac{p}{2-p}} C(u_0) =C(u_0)  \left(\frac{|x_0|}{4}\right) ^{N-\frac{p}{2-p}}.
$$
Then
\begin{align*}
 u(x_0, t) &\le  \kappa_1
C(u_0)^{p\beta}   \left(\frac{|x_0|}{4}\right) ^{p\beta(  N -\frac{p}{2-p})}
 t^{-N\beta} + \kappa_2 \,4^{\frac{p}{2-p}} \,\left(\frac{t}{|x_0|^{p}}\right)^{\frac{1}{2-p}} \\
 & =  2^{\frac{2p}{2-p}}\left(  \kappa_1
C(u_0)^{p\beta}   t^{-N\beta} + \kappa_2 \, t^{\frac{1}{2-p}}
  \right)   \, |x_0|^{-\frac{p}{2-p}}.
\end{align*}

Now, we begin with the proof of the theorem. Let $t_0 >0$. We have proved that
\begin{equation}\label{inequality.outside.0}
u(x,t_0/2) \le A |x|^{-\frac{p}{2-p}}, \quad \forall x\in \RN,
\end{equation}
with $$A:=2^{\frac{2p}{2-p}} \left( \kappa_1
C(u_0)^{p\beta}
 (t_0/2)^{-N\beta} + \kappa_2  \, (t_0/2)^{\frac{1}{2-p}}  \right).$$
Let $\widetilde{u_0}(\cdot):=u(\cdot, t_0/2)$ and $\widetilde{u}(x,t)$ be the weak solution to problem \eqref{PLE} with data $\widetilde{u_0}$. Note that, by the semigroup structure
of solutions to the \eqref{PLE} problem, we have that $u(\cdot, t_0/2+t)=\widetilde{u}(\cdot,t).$
We apply Theorem \ref{ThmUpperBound} for data $\widetilde{u_0}$ and time $t_0/2$. Therefore there exist $\tau_2, \, M_2$ such that for any $x \in \RN$, $t \in (t_0/2 , \infty)$ we
have the following upper bound
\begin{equation}\label{upperbound:otravez}
u(x,t+t_0/2)=\widetilde{u}(x,t) \le \mathcal{B}(x,t+\tau_2;M_2),
\end{equation}
where $\tau_2 =\tau_2(t_0/2,A,t_s)$, $M_2=M_2(t_0/2,A,\tau_2)$ , $t_s =C_5  R_0^{1/\beta}.$
Let us denote $s=t+t_0/2$. Then \eqref{upperbound:otravez} rewrites as:
$$
u(x,s) \le \mathcal{B}(x,s-t_0/2+\tau_2;M_2), \quad \forall s >t_0,\, x\in \RN.$$
Let
$$\tau_2^*= \tau_2-t_0/2.
 $$
 By \eqref{def:tau2} it follows that $\tau_2^*\ge 0.$ Then we conclude that
\begin{equation*}
u(x,s) \le \mathcal{B}(x,s+\tau_2^*;M_2), \quad \forall s >t_0,\, x\in \RN.
\end{equation*}
\end{proof}

Now, we begin the proof of the lower bound.
The main tool to prove Theorem \ref{lower.estimate.theorem}  is the following Proposition, which has an interest of its own.
\begin{proposition}\label{lower.bound.proposition}
Let $N\ge 1$ and $\frac{2N}{N+1}<p<2.$ Let $u$ be a weak solution to Problem \eqref{PLEcauchy} corresponding to an initial datum  $ u_0\in L^1(\RN)$, let $R>0$ and $x_0\in \RN$. Then
 \begin{equation}\label{lower.bound}
 \inf_{x\in B_R(x_0)} u(x,t) \ge \underline{C} \frac{M_R}{R^N} \left\{ \begin{array}{ll}
  \left(\frac{t}{t_c}\right)^{-N\beta} &\text{for }t\ge t_c, \\[2mm]
  \left(\frac{t}{t_c}\right)^{1/(2-p)} &\text{for }t\le t_c,
    \end{array}
    \right.
\end{equation}
where
\begin{equation}\label{critical.time}
t_c:=\kappa\,M_R^{2-p}\,R^{\frac{1}{\beta}}\quad\text{and}\quad M_R:=\int_{B_R(x_0)}u_0\,dx\,.
\end{equation}
\end{proposition}

We will first give the proof of Theorem \ref{lower.estimate.theorem} and we postpone the proof of  Proposition  \ref{lower.bound.proposition} to the end of this section.

\begin{proof}[Proof of Theorem~\ref{lower.estimate.theorem}]
Let us explain first the strategy of the proof. We will prove first that inequality~\eqref{lower.estimate.inequality} holds at time $t=t_c$ and then conclude discussing the two
different cases, namely $t_0 \geq t_c$ and  $0< t_0 < t_c$. In what follows, we will sometimes denote $B_{R_0}(0)$ by $B_{R_0}$ .

\noindent\textit{Proof of inequality~\eqref{lower.estimate.inequality} for time $t=t_c$}. More precisely we will prove that
\begin{equation}\label{lower.ghp.1}
u(x, t_c) \geq \mathcal{B}(x, t_c-\underline{\tau}; \underline{M})\,,\quad\mbox{for all}\,\,x\in\RN\,,
\end{equation}
where
\[
\underline{\tau}=a\,t_c\,\quad\mbox{and}\quad\underline{M}=b\,M_{R_0}\,,
\]
with $a \in (0,1)$ and $b>0$ to be chosen later, $t_c$ as in~\eqref{critical.time} and $M_{R_0}=\|u_0\|_{L^1(B_{R_0}(0))}$. We split the proof  in different steps. First, we find
conditions \eqref{condition.lower.bounds.1}  on $a,b $ so that~\eqref{lower.ghp.1} holds in $|x| \leq R_0$. Next, we will find sufficient conditions on $a, b$ so
that~\eqref{lower.ghp.1} holds in $|x| \geq R_0$.

\noindent\textit{Condition inside a ball.} The idea is to make use of the lower bound  \eqref{lower.bound} and put a Barenblatt solution from below inside the ball $B_{R_0}(0)$. By
inequality~\eqref{lower.bound} at time $t=t_c$ we can write:
\begin{equation}\label{lower.bound:tc}
\inf_{x\in B_{R_0}(0)}  u(x,t_c) \ge \underline{C}\, \frac{M_{R_0}}{{R_0}^{N}}.
\end{equation}
Therefore, to prove~\eqref{lower.ghp.1} in $B_{R_0}(0)$, it is sufficient to find conditions on $a, b$ such that the following  inequality holds:
\begin{equation}\label{inequality.inside.ball}
 \underline{C}\, \frac{M_{R_0}}{{R_0}^{N}} \geq \frac{b^{p\,\beta}M_{R_0}}{b_1^\frac{p-1}{2-p}(1-a)^{N\beta}\kappa^{N\beta}{R_0}^N}= \sup_{x \in B_{{R_0}}(0)} \mathcal{B}(x,
 t_c-\underline{\tau}; \underline{M})\,,
\end{equation}
where $\underline{C}$ is as in~\eqref{lower.bound} and $\kappa$ is as~\eqref{critical.time}. It is easily seen that the former is implied by requiring
\begin{equation}\label{condition.lower.bounds.1}
  b^{p\,\beta} \leq \kappa^{N\,\beta}\, \underline{C}\, b_1^{\frac{p-1}{2-p}}\, (1-a)^{N\,\beta}\,.
\end{equation}
Note that, by inequality \eqref{lower.bound:tc}, the first term in \eqref{inequality.inside.ball} is bounded above by $\inf_{x \in B_{R_0}}u(x,t_c)$, therefore we obtain that
\begin{equation*}
\inf_{x \in B_{R_0}}u(x, t_c) \geq     \sup_{x \in B_{{R_0}}(0)} \mathcal{B}(x, t_c-\underline{\tau}; \underline{M})\,.
\end{equation*}
Inequality \eqref{lower.ghp.1} is then proved for any $|x| \leq  R_0$.

\noindent\textit{Condition outside a ball.} We want to find suitable conditions on $a, b$ such that \eqref{lower.ghp.1} holds in the outer region $|x|> R_0$. Such an inequality will
be deduced by applying the comparison on the parabolic boundary of $Q= \RN\setminus B_{R_0}(0)\times \left(\underline{\tau}, t_c\right)$, namely $\partial_p Q=\Big\{ \RN\setminus B_{R_0}(0)\times
\big\{\underline{\tau}\big\} \Big\} \bigcup \Big\{ \big\{x \in \RN: |x|=R_0\big\} \times\left(\underline{\tau}, t_c\right)\Big\}$.

It is clear that $u(x,\underline{\tau}) \geq \mathcal{B}(x,0; \underline{M})=\underline{M}\delta_0(x)=0$, for any $|x| \geq R_0$, hence we just need to prove that
\begin{equation*}
u(x, t) \geq \mathcal{B}(x, t-\underline{\tau}; \underline{M})\quad\mbox{for any}\quad|x|=R_0,\, t \in \left(\underline{\tau}, t_c\right).
\end{equation*}
We have that $\displaystyle \inf_{|x|=R_0}u(x,t)\ge \inf_{x\in B_{R_0}}u(x,t)$, since it is known that solutions to~\eqref{PLEcauchy} are continuous, see~\cite{DiBenedetto1990}. By
the lower bound \eqref{lower.bound} for $t\le t_c$, we have that
$$\inf_{\substack{ |x|=R_0,\\ t \in \left(\underline{\tau}, t_c\right)}} u(x,t) \ge
\inf_{\substack{|x|=R_0,\\ t \in \left(\underline{\tau}, t_c\right)}} \underline{C} \frac{M_{R_0}}{R_0^N} \left(\frac{t}{t_c}\right)^{\frac{
1}{2-p}} \ge \underline{C} \frac{M_{R_0}}{R_0^N} a ^{\frac{
1}{2-p}}.$$
Also, notice that
\begin{equation*}\begin{split}
\sup_{\substack{t \in \left(at_c, t_c\right),\\ |x|=R_0}} & \mathcal{B}(x, t-\underline{\tau}; \underline{M})= \sup_{\substack{t \in \left(at_c, t_c\right),\\ |x|=R_0}}
(t-at_c)^{\frac{1}{2-p}} \left[ b_1 \frac{(t-at_c)^{\frac{\beta p}{p-1}}}{\underline{M}^{(2-p)\frac{\beta p}{p-1}}} + b_2R_0^{\frac{p}{p-1}}\right]^{-{\frac{p-1}{2-p}}} \\
&\le \frac{(1-a)^\frac{1}{2-p}\,t_c^\frac{1}{2-p}}{\left[b_2 {R_0}^\frac{p}{p-1}\right]^\frac{p-1}{2-p}}
= \frac{(1-a)^\frac{1}{2-p}\,t_c^\frac{1}{2-p}}{b_2^\frac{p-1}{2-p} R_0 ^{\frac{p}{2-p}} }.
\end{split}\end{equation*}
Thus, we impose the condition:
$$ \underline{C}\, a^\frac{1}{2-p}\, \frac{M_{R_0}}{R_0^N}\, \ge \frac{(1-a)^\frac{1}{2-p}\,t_c^\frac{1}{2-p}}{b_2^\frac{p-1}{2-p} R_0 ^{\frac{p}{2-p}} }.$$
Recalling that $t_c=\kappa\, M_{R_0}^{2-p}\, R_0^{1/\beta}$, this is equivalent to imposing:
$\displaystyle \underline{C}^{2-p} b_2^{p-1 } \frac{1}{\kappa} \, a \ge (1-a) ,$
that is
\begin{equation}\label{def:a}
 a\ge \frac{1 }{1+\underline{C}^{2-p} b_2^{p-1 }\kappa^{-1}}.
 \end{equation}
 Thus take any $a<1$ satisfying \eqref{def:a} and then $b$ satisfying \eqref{condition.lower.bounds.1}, for instance:
 \begin{equation}\label{choice.a.b}
a=  \frac{1 }{1+\underline{C}^{2-p} b_2^{p-1 }\kappa^{-1}},  \quad
b^{p\,\beta} = \kappa^{N\,\beta}\, \underline{C}\, b_1^{\frac{p-1}{2-p}}\, (1-a)^{N\,\beta}\,.
\end{equation}
This concludes the proof of~{\eqref{lower.ghp.1}}.

\noindent\textit{Case $t_0\ge t_c$.}  Since inequality~\eqref{lower.estimate.inequality} holds for $t=t_c$, then, by the Comparison Principle, it holds for any $t\ge t_c$.

\noindent\textit{Case $0<t_0 < t_c$.}  As already mentioned, we only need to prove inequality~\eqref{lower.estimate.inequality} at time $t_0$, the result will then follow using the
Comparison Principle. From the Benilan-Crandall-type estimate~\eqref{BC:decreasing} in the monotonicity form it follows that $\displaystyle u(x, t_0) \geq u(x, t_c)
\left(\frac{t_0}{t_c}\right)^\frac{1}{2-p}\,. $ for any $0 < t_0 < t_c $. We already know, by our previous procedure, that $\displaystyle u(x, t_c)\ge
\mathcal{B}(x,t_c-\underline{\tau}; \underline{M})$ with $\underline{\tau}, \underline{M}$ as in~\ref{choice.case.i}. Combining these estimates and using the self-similarity
properties of the Barenblatt profile (recall~\eqref{transf} and~\eqref{self.similarity.relation}) we have:
\begin{equation*}\begin{split}
u(x, t_0) &\geq u(x, t_c) \left(\frac{t_0}{t_c}\right)^\frac{1}{2-p} \geq \left(\frac{t_0}{t_c}\right)^\frac{1}{2-p}  \mathcal{B}(x,t_c-\underline{\tau}; \underline{M})\\
&=\left(\frac{t_0}{t_c}\right)^\frac{1}{2-p}  (t_c-at_c)^{\frac{1}{2-p}} \left[ b_1 \frac{(t_c-at_c)^{\frac{\beta p}{p-1}}}{\underline{M}^{(2-p)\frac{\beta p}{p-1}}} +
b_2|x|^{\frac{p}{p-1}}\right]^{-{\frac{p-1}{2-p}}} \\
&=\mathcal{B}\left(x,t_0(1-a); \underline{M}\, \left(\frac{t_c}{t_0}\right)^{-\frac{1}{2-p}} \right).
\end{split}\end{equation*}
The proof is now finished.
\end{proof}
Let us conclude this section with the proof of Proposition~\ref{lower.bound.proposition}.

\begin{proof}[Proof of Proposition~\ref{lower.bound.proposition}]
Without loss of generality we can assume that $u_0$ is supported in $B_R(x_0)$ and that $x_0=0$. If it is not the case, then we define the function $v_0=u_0 \chi_{B_R(x_0)}$, where
$\chi_{B_R(x_0)}=1$ on $B_R(x_0)$ and $\chi_{B_R(x_0)}=0$ outside $B_R(x_0)$, and let $v(x,t)$ be the solution to Problem \eqref{PLEcauchy} with $v_0$ as its initial data. Then, by
the Comparison Principle, $u(\cdot,t)\ge v(\cdot,t)$ for any $t\ge0$.  Therefore, proving firstly the result for $v(x,t)$ will imply  the desired result for $u(x,t)$. As well,
taking into account the scaling properties of the equation, we can assume that
\begin{equation}\label{mass:1}
\int_{\RN}u_0(x) dx= \int_{B_R(x_0)}u_0(x)dx=M_R=1.
\end{equation} We shall divide the proof in two steps. In what follows we will denote $B_R(0)$ by $B_R$.

\noindent\textit{Case $t<t_c$}. We observe that for $t< t_c$ inequality~\eqref{lower.bound} is a consequence of the Benilan-Crandall estimate~\eqref{BenilanCrandall} combined
with~\eqref{lower.bound} at $t=t_c$.    Indeed, using the monotonicity \eqref{BC:decreasing}, we have
\[
 t^{-\frac{1}{2-p}} u(x,t) \ge t_c^{-\frac{1}{2-p}} u(x,t_c)\,,\quad\forall t\le t_c\,,
\]
where we have used the fact that $M_R=1$. By combining with the lower bound at time $t_c$, in which case it just says that $u(x,t_c)\ge \frac{\underline{C}}{R^N}$, we have that
\[
u(x,t) \ge  (t/ t_c)^{\frac{1}{2-p}} u(x,t_c) \ge\, \frac{\underline{C}}{R^N}  (t/ t_c)^{\frac{1}{2-p}}, \quad \forall x\in B_R(0)\,.
\]

\noindent\textit{Case $t\ge t_c$. }
Let us define an auxiliary time
\begin{equation}\label{def:t*}
t_\star= \tilde{\kappa}\,R^\frac{1}{\beta}
\end{equation}  where
$$\tilde{\kappa}=\max\{ (\frac{\omega_N}{N}\,2^{N+1}C_1)^{\frac{1}{N\beta}}\, ,\,
( \kappa_2 \,4^{\frac{p}{2-p}}\, \omega_N \, (2-p)\beta )^{-(2-p)} 2^{\frac{1}{\beta}} \},$$
with $C_1$ being the constant of inequality~\eqref{smoothingEffect}, $\omega_N$ being the surface are of the sphere $\mathbb{S}^{N-1}$, $\omega_N\,N^{-1}$ being the volume of the ball of
radius $1$ in $\RN$ and $\kappa_2$ the constant from~\eqref{local.smoothing.effect}. We shall explain in what follows this choice for $t_{\star}$.

 We will first prove an initial lower bound for $u(0, t_\star)$ and then we shall generalize it to inequality~\eqref{lower.bound}.
Let $r>0$ and $t\ge t_\star$, by mass conservation, we have that
\[
1=\int_{\RN}u_0(x)dx = \int_{\RN}u(x,t)dx =\int_{B_{2R}}u(x,t)dx +\int_{B_{2R+r}\setminus B_{2R}}u(x,t)dx +\int_{\RN \setminus B_{2R+r}}u(x,t)dx .
\]
For the first integral we apply the smoothing effect~\eqref{smoothingEffect}, which is this context is equivalent to $u(x,t) \le C_1 t^{-N\beta} $:
 \begin{equation*}\int_{B_{2R}}u(x,t)dx \le C_1 t^{-N\beta} |B_{2R}| =  C_1  t^{-N\beta} \frac{\omega_N}{N} (2R)^{N}= C_1  \frac{\omega_N}{N}\frac{(2R)^{N}}{t^{N\beta}}.
 \end{equation*}
 For the second integral we apply the Aleksandrov reflection principle: under the current assumptions we have that
\begin{equation}\label{aleksandroff.principle}
u(t, x)\le u(t, 0)\,,\quad\forall\, t>0\quad\mbox{and}\quad\forall\,x\in B_{2R+r}\setminus B_{2R}\,.
\end{equation}
For a proof of this result see \cite[Prop.A.1,pp.425]{BV2006} and recently \cite{BDNS2020Suplement}. The latter contains a proof  for the Fast Diffusion Equation ($u_t=\Delta u^m$,
$m<1$), which is exactly the same as in present case since it is based only on two ingredients: comparison principle and invariance of the equation under translations and
reflections. See also~\cite[Proposition 2.24]{Galaktionov2004} and~\cite[Section 9.6.2]{Vazquez2007} where the principle is discussed in details.

From~\eqref{aleksandroff.principle} we deduce that
 \begin{equation*}
\int_{B_{2R+r}\setminus B_{2R}}u(x,t)dx \le |B_{2R+r}\setminus B_{2R}|\, u(0,t)= \frac{\omega_N}{N}\big((2R+r)^N-(2R)^N\big)\,u(0,t).
  \end{equation*}

For the third integral we argue in the following way: for any $x_0 \in \RN\setminus B_{2R+r}$ we use the local smoothing effect~\eqref{local.smoothing.effect} on the ball
$B_{R_0}(x_0)$ with $R_0=|x_0|/2$. Since $u_0$ is supported in $B_R(0)$, then $\int_{B_{R_0}(x_0)}u_0dx=0$ and so
\[
u(x_0,t)\le \kappa_2\,4^{\frac{p}{2-p}} \left(\frac{t}{|x_0|^p}\right)^\frac{1}{2-p}\,,\quad\forall\,t>0\quad\mbox{and}\quad\forall\,x_0\in \R^N\setminus B_{2R+r}\,.
\]
Applying the above inequality we get
\begin{align*}
    \int_{\RN \setminus B_{2R+r}}u(x,t)dx & \le  \kappa_2 \,4^{\frac{p}{2-p}}\,\int_{\RN \setminus B_{2R+r}} \left(\frac{t}{|x|^{p}}\right)^{\frac{1}{2-p}}\\
    & \le \kappa_2 \,4^{\frac{p}{2-p}}\,t^{\frac{1}{2-p}}\,\omega_N \,  \int_{2R+r}^{\infty}  \rho^{-{\frac{p}{2-p}} +N-1}d\rho \\
    & \le \kappa_2 \,4^{\frac{p}{2-p}}\,t^{\frac{1}{2-p}}\, \omega_N \, (2-p)\beta \, (2R+r)^{-\frac{1}{\beta(2-p)}} \,.
\end{align*}
Combining the previous estimates all together we arrive at:
\begin{equation}\label{def.f}
u(0, t)\ge \frac{\left(B(t) - A(t)\,(2R+r)^{-\frac{1}{\beta(2-p)}} \right)}{((2R+r)^N-(2R)^N)}\,\,\frac{N}{\omega_N} =: f(r,t)\,\, \frac{N}{\omega_N}\,,
\end{equation}
where
\begin{equation*}
B(t) = 1 - C_1\,\frac{\omega_N}{N}\,\frac{(2R)^{N}}{t^{N\beta}}\,,\quad\mbox{and}\quad  A(t)=t^{\frac{1}{2-p}}\,\kappa_2 \,4^{\frac{p}{2-p}}\, \omega_N \, (2-p)\beta\,.
\end{equation*}
Notice that $B(t)$ is a strictly-increasing function of time such that $\lim_{t\to \infty}B(t)=1$ and
$B(t)<1$ for all $t>0$.  Let $t_1$ such that $B(t_1)=1/2$. Also, notice that $A(t)$ is strictly increasing with  ${\lim_{t\to \infty}A(t)=+\infty}.$ Let $t_2$ such that
$A(t_2)\,(2R)^{-\frac{1}{\beta(2-p)}}  = 1$.
Thus
$$t_1=(2^{N+1}C_1\,\frac{\omega_N}{N})^{\frac{1}{N\beta}}\, R^{\frac{1}{\beta}}, \quad t_ 2 =\left(\kappa_2 \,4^{\frac{p}{2-p}}\, \omega_N \, (2-p)\beta \right)^{p-2}
2^{\frac{1}{\beta}}
R^{\frac{1}{\beta}}.$$
We take $t_{\star}=\max(t_1,t_2)$, which is definition \eqref{def:t*}. We note that for $t\ge t_\star$ we have that $1>B(t)\ge B(t_\star)\ge 1/2$ and
$A(t_\star)\,(2R)^{-\frac{1}{\beta(2-p)}}\ge1$.  This choice of $t_\star$ guarantees that
\begin{equation}\label{Bt-At}
B(t)-A(t) \,(2R)^{-\frac{1}{\beta(2-p)}}  <0 \quad \text{for all } t\ge t_\star.
\end{equation}
For any $t\ge t_\star$ the function $f(r,t)$ is continuous in $r$ and by \eqref{Bt-At} we obtain the following limit values:
\[
\lim_{r\rightarrow 0^+}f(r, t)=-\infty\,\quad\mbox{and}\quad \lim_{r\rightarrow \infty}f(r, t) =0\,,\quad\forall\, t\ge t_\star\,.
\]
Let us now fix $t\ge t_\star$ and let us consider $f(r)=f(r,t)$ as a function of $r$. Since $f(r)$ is sign changing we conclude that it has at least one local maximum, which we call
$r_m$ (such $r_m$ depends on $t$, however here $t$ is fixed, so we will drop such a dependence since it will cause no harm in what follows). At such point we have that $f'(r_m)=0$,
which translates into the following
condition on $r_m$
\begin{equation}\label{condition.maximum}
\frac{A(t)}{\beta(2-p)}\left[\left(2R+r_m\right)^N-(2R)^N\right] = \left[B(t)\left(2R+r_m\right)^{\frac{1}{\beta(2-p)}}-A(t)\right]\,N\,(2R+r_m)^N.
\end{equation}
By condition~\eqref{condition.maximum} we have that
\[
f(r_m) = \frac{A(t)}{N\,\beta\,(2-p)\,\left(2R+r_m\right)^{\frac{p}{2-p}}}\, .
\]
By simple algebraic manipulation, we deduce from~\eqref{condition.maximum} that
\[
\frac{1}{\left(2R+r_m\right)^{\frac{p}{2-p}}} \ge \left(\frac{2-p}{p}\right)^{p\beta}\left(\frac{N\,B(t)}{A(t)}\right)^{p\beta}\,,
\]
and therefore
\begin{equation}\label{almost.last.estimate}
f(r_m)\ge \frac{(N(2-p))^{N\beta(2-p)}}{\beta\,p^{p\beta}}\,\frac{B(t)^{p\beta}}{A(t)^{N\beta(2-p)}}\,.
\end{equation}
Finally we notice that $B(t)\ge B(t_\star)\ge \frac{1}{2}$ for any $t\ge t_\star$ and that $A(t)^{N\beta(2-p)}= \underline{c} t^{N\beta}$. Thus, by
combining~\eqref{almost.last.estimate} with~\eqref{def.f} we obtain
\begin{equation*}
u(0, t)\ge \frac{\underline{\kappa}}{R^N} \left(\frac{t_\star}{t}\right)^{N\beta}\,,\quad\forall\,t\ge t_\star\,,
\end{equation*}
where
\[
\underline{\kappa}= \frac{N^{N\beta(2-p)+1}}{\tilde{\kappa}^{N\beta}\,\omega_N\,\left(\beta\,2p\right)^{p\beta}\left(\kappa_2\,4^\frac{p}{2-p} \omega_N \right)^{N\beta(2-p)}}.
\]

Now we shall pass from the center of $B_{2R}(0)$ to the infimum of $u(\cdot,t)$ on the ball $B_R(0)$.  Let $y\in B_{2R}(0)$ and define
$t_\star(y)=\tilde{\kappa}\,(4R)^{\frac{1}{\beta}}$. We
apply the above mentioned procedure to the function $u$ in the ball $B_{4R}(y)$ (notice than $u_0$ is supported in such a ball) and we get the following inequality
\begin{equation}\label{inequality.uniform}
u(y, t)\ge \frac{\underline{\kappa}}{(4R)^{N}} \left(\frac{t_\star(y)}{t}\right)^{N\beta}\,,\quad\mbox{for any}\quad t\ge t_\star(y)\,.
\end{equation}
We notice that for any $y, y_0 \in B_{2R}(0)$ the times $t_\star(y)=t_\star(y_0)$ (in other words for any $y \in B_{2R}(0)$ such a time is equal to a constant depending only on the
radius of $B_{2R}(0)$ and therefore inequality~\eqref{inequality.uniform} is uniform in $y \in B_{2R}(0)$. Taking the infimum (in $y \in B_{2R}(0)$) in~\eqref{inequality.uniform} we
get inequality~\eqref{lower.bound} for any $t\ge t_c$ where $t_c$ and the constant $C$ of~\eqref{lower.bound}  have the following expression (recall that we assumed \eqref{mass:1}):
\[
t_c:=4^{\frac{1}{\beta}}\,\tilde{\kappa}\,R^\frac{1}{\beta}\quad\mbox{and}\quad \underline{C}= \underline{\kappa}\, 4^{-N(\frac{\beta+1}{\beta})}\,.
\]
The proof is concluded once one re-scales back to the original variables.
\end{proof}

\section{Large time asymptotic behavior}\label{Section:asymptotics}
In this section we supply the proof the \emph{convergence in relative error} of solutions to~\eqref{PLEcauchy} with  data $0\le u_0 \in \X\setminus\{0\}$  to the Barenblatt porfile with mass $M>0$.
\begin{theorem}\label{convergence.relative.error.thm}
Let $N\ge 1$ and $\frac{2N}{N+1}<p<2.$ Let $u$ be a weak solution to Problem \eqref{PLEcauchy} corresponding to an initial datum  $0\le u_0\in \X\setminus\{0\}$. Then
\end{theorem}
\begin{equation*}
\lim\limits_{t\rightarrow \infty} \Big\|\frac{u(\cdot,t)}{\mathcal{B}(\cdot,t; M)}-1\Big\|_{L^\infty(\RN)}= 0\,,\quad\mbox{where}\quad M=\|u_0\|_{L^1(\RN)}\,.
\end{equation*}
The proof is based on two ingredients: the convergence to the Barenblatt profile in the $L^\infty$ norm (see~\eqref{AsympBehLinf} of Theorem~\eqref{Th:AsympBehav} that we prove at the end of this section) and the GHP in the form of upper bounds of Theorem~\ref{ThmUpperBound2} and lower bound of Theorem~\ref{lower.estimate.theorem}. The latter is needed to control the relative error locally while the former gives us a way to control the tail
of the solution $u$.

\begin{proof}
We will prove that for any $\varepsilon>0$ there exists $t_\varepsilon>0$ such that
\begin{equation}\label{claim.inequality}
\Big\|\frac{u(\cdot,t)}{\mathcal{B}(\cdot,t; M)}-1\Big\|_{L^\infty(\RN)}< \varepsilon\,,\quad\mbox{for all}\,\, t \ge t_\varepsilon\,.
\end{equation}
In the proof we make use of the following interior and exterior cones, namely when $\{|x|\le C\, t^\beta\}$ and $\{|x|\ge C\, t^\beta\}$. We emphasize that splitting
$\RN\times\left(0, \infty\right)$ in such a way is equivalent to work with interior and exterior of a ball $\{|y|<C\}$ into the self-similar variable $y=xt^{-\beta}$. We begin by
proving the following \textit{Claim}.

\noindent\textit{Claim}. Let $C>0$, then under the running assumptions the following limit holds
\begin{equation}\label{convergence:interiorcones}
\lim\limits_{t\rightarrow \infty} \sup\limits_{\{|x|\le C\,t^\beta\}}\Big|\frac{u(x,t)}{\mathcal{B}(x,t; M)}-1\Big|=0\,.
\end{equation}

\noindent\textit{Proof of Claim.}
We can express the relative error as:
\begin{align*}
\Big|\frac{u(x,t)}{\mathcal{B}(x,t; M)}-1\Big| &= t^{N\beta} \Big|u(x,t)-\mathcal{B}(x,t;M) \Big|  \frac{1}{t^{N\beta} \mathcal{B}(x,t;M)  } \\
&= t^{N\beta} \Big|u(x,t)-\mathcal{B}(x,t;M) \Big|  \frac{1}{\mathcal{B}(xt^{-\beta},1;M)  }=\frac{t^{N\beta}\Big|u(x,t)-\mathcal{B}(x,t;M)
\Big|}{M^{p\beta}\,F_1(|xt^{-\beta}|M^{(2-p)\beta})}\,,
\end{align*}
where we used the self-similarity of the Barenblatt solution, see~\eqref{self.similarity.relation}. Thus, in cones of the form $\{ | xt^{-\beta}| \le C \}$, we have the bounds:
\begin{equation*}
\sup_{\{| xt^{-\beta}| \le C\} }  \Big|\frac{u(x,t)}{\mathcal{B}(x,t; M)}-1\Big|
\le M^{-p\beta}\,\left(b_1+b_2\,C^\frac{p}{p-1}\,M^\frac{(2-p)p\beta}{p-1}\right)^\frac{p-1}{2-p}\,t^{N\beta} \Big\|u(x,t)-\mathcal{B}(x,t;M)    \Big\|_{L^{\infty}(\RN)}\,,
\end{equation*}
where we have used the expression of $F_1$, see~\eqref{F1}. Thus, since in the above inequality the right hand side tends to $0$ as $t\to \infty$, we
obtain~\eqref{convergence:interiorcones}.

We analize now what happens in exterior cones of the form $\{|x|\ge C t^{\beta}\}$. Let us observe that, as already described in \cite{Vaz2003,Carrillo2003}, for $|x|\sim\infty$,
the behaviour of $\mathcal{B}(x,t;M)$ does not depend on the mass $M$. Indeed, by using~\eqref{F1} and~\eqref{self.similarity.relation}, we have that
\begin{equation}\label{barenblatt.behaviour.infty}
\mathcal{B}( x,t; M)\sim \frac{t^\frac{1}{2-p}}{b_2^\frac{p-1}{2-p}\,|x|^\frac{p}{2-p}}\,,\quad\mbox{as}\,\,|x|\rightarrow \infty\,.
\end{equation}
Let us fix $\varepsilon>0$ and $t_0>0$. By applying inequalities~\eqref{lower.estimate.theorem} and~\eqref{upperboundXdata}, we know that for any $t\ge t_0$ we have
\[
\frac{ \mathcal{B}(x,t-\underline{\tau}; \underline{M})} {\mathcal{B}( x,t; M)}  \le u(x,t) \le \frac{\mathcal{B}(x,t+\overline{\tau}; \overline{M})}{\mathcal{B}( x,t; M)}\,.
\]
for some $\underline{\tau}, \overline{\tau}, \underline{M}, \overline{M}>0$. This is equivalent to writing:
$$
\frac{  (t-\underline{\tau})^{-N\beta}  \mathcal{B}(x(t-\underline{\tau})^{-\beta}  ,1; \underline{M})} { t^{-N\beta} \mathcal{B}( x t^{-\beta} ,1; M)}  \le
\frac{u(x,t)}{\mathcal{B}( x,t; M)} \le
\frac{ (t+\overline{\tau})^{-N\beta}  \mathcal{B}(x(t+\overline{\tau})^{-\beta},1; \overline{M})}{ t^{-N\beta} \mathcal{B}( x t^{-\beta} ,1; M)}\,.
$$
For large argument $|x t^{-\beta} | \to \infty$, the Barenblatt solution behaves as specified in \eqref{barenblatt.behaviour.infty}. Notice that also $|x (t+\overline{\tau})
^{-\beta} | \to \infty$ and
 $|x (t-\underline{\tau}) ^{-\beta} | \to \infty$. Thus as $|x t^{-\beta} | \to \infty$ we have:
 $$ \left(  \frac{t-\underline{\tau} }{t}  \right) ^{\frac{1}{2-p}} \le \lim\limits_{|xt^{-\beta}|\rightarrow\infty} \frac{u(x,t)}{\mathcal{B}( x,t; M)} \le \left(
 \frac{t+\overline{\tau} }{t}  \right) ^{\frac{1}{2-p}} .$$
Notice that both $\displaystyle \left(  \frac{t-\underline{\tau} }{t}  \right)^{\frac{1}{2-p}}\,, \left(  \frac{t+\overline{\tau} }{t}  \right)^{\frac{1}{2-p}} \rightarrow 1 $ as
$t\rightarrow \infty$. We conclude that there exist $C'_\varepsilon, t'_\varepsilon>0$ such that
\begin{equation}\label{inequality.claim.1}
1-\varepsilon \le \frac{u(x,t)}{\mathcal{B}(x,t; M)} \le 1+\varepsilon\,,\quad\mbox{for all}\,\, t\ge t'_\varepsilon\,,\quad\mbox{and}\,\,x \in \{|x|\ge C'_{\varepsilon}\,
t^{\beta}\}\,.
\end{equation}

Also, from \eqref{convergence:interiorcones}, there exists $t''_\varepsilon>0$ such that
\begin{equation}\label{inequality.claim.2}
\sup\limits_{\{|x|\le C'_\varepsilon\,t^\beta\}}\Big|\frac{u(x,t)}{\mathcal{B}(x,t; M)}-1\Big| \le \varepsilon\,,\quad\mbox{for all}\,\, t\ge t''_\varepsilon\,.
\end{equation}

By combining~\eqref{inequality.claim.1} with~\eqref{inequality.claim.2} we obtain~\eqref{claim.inequality}. The proof is concluded.
\end{proof}

For the sake of completeness we prove the convergence of $u$ to the Barenblatt solution with mass $M$. We refer to \cite[Thm.1.2]{Vaz2020} when taking the particular case $s=1$.
Very similar results for the Fast Diffusion Equation/Porous Medium Equation have been proven in \cite{Friedman1980} and in \cite{Vaz2003}. In \cite{KaminVazquez1988} the authors
prove the above statement in the case $p>2$ using the so called "4 steps method" (see also \cite{Vaz2003} for a detailed account of this method). For completeness, in Theorem
\ref{Th:AsympBehav} we adapt their proof  to the case $\frac{2N}{N+1}<p < 2$ with some modifications.

\begin{theorem}\label{Th:AsympBehav}
Let $N\ge 1$ and $\frac{2N}{N+1}<p < 2$.  Let $u$ be the solution to the Cauchy problem \eqref{PLEcauchy} with mass $\int_{\RN}u_0(x) dx =:M$. Then
\begin{equation}\label{AsympBehL1}
  \| u(\cdot,t) - \mathcal{B}(\cdot,t;M) \|_{L^1(\RN)} \to 0 ,\quad \text{as } t\to \infty,
\end{equation}
\begin{equation}\label{AsympBehLinf}
 t^{N\beta}  \| u(\cdot,t) - \mathcal{B}(\cdot,t;M) \|_{L^\infty(\RN)} \to 0 ,\quad \text{as } t\to \infty,
\end{equation}
where $\mathcal{B}(x,t;M)$ is the Barenblatt solution with mass $M$.
\end{theorem}
\begin{proof}
The proof follows the 4-steps method developed by Kamin and V\'azquez \cite{KaminVazquez1988}. We also use some ideas from \cite{STVsurvey}.

Step 1.\textit{ Rescaling.} We use the mass-preserving scaling transformation $u_\lambda(x,t)= \lambda^\alpha u(\lambda^\beta x, \lambda t)$, for $\lambda>0.$ Then one can easily
check that $u_\lambda$ is a solution to \eqref{PLEcauchy} with data $u_{0,\lambda}= \lambda^\alpha u_0(\lambda^\beta x)$. The total mass is the same of $u$ and $u_{\lambda}$.
Indeed, $\int_{\RN}u_{\lambda}(x,t)dx= \int_{\RN}u_{0,\lambda}(x)dx =\int_{\RN}u_{0}(x)dx =M.$ Since $u_0\ge 0$ then $u>0$ for all $t>0$ and thus $u_{\lambda}>0$ for all $x\in \RN$,
$t>0.$

Step 2. \textit{Energy estimates.} One can easily prove that for any solution to the Cauchy problem \eqref{PLEcauchy}, the following holds:
$$ \int_{\RN} u^2(x,t_1) (x) dx   =  \int_{\RN} u^2 (x,t_ 2) dx + 2 \int_{t_1}^{t_2}  \int_{\RN} | \nabla u(x,t)|^p dx \, dt,\quad  \text{for } 0\le t_1<t_2<\infty.$$
Thus, if $u_0\in L^2(\RN)$ then $u \in L^2((t_1,t_2):L^2(\RN)) \cap L^p((t_1,t_2):W^{1,p} (\RN)).$
The smoothing effect (\cite{JLVSmoothing}) is as follows:
$$|u(x,t) | \le c(p,N) \|u_0\|_{L^1}^{p\alpha/N} t^{-\alpha}, \quad \forall x \in \RN, \, t>0.$$
In terms of $u_\lambda$ this gives
\begin{equation}\label{smoothing:ulambda}
|u_\lambda(x,t) | \le \lambda^{\alpha} c(p,N) M^{p\alpha/N} (\lambda t)^{-\alpha}=c(p,N) M^{p\alpha/N} t^{-\alpha}.
\end{equation}
Thus the family $(u_\lambda)_{\lambda>0} \in L^\infty(\RN \times (t_1,t_2))$ for every $0<t_1<t_2 <\infty$.
Also, combining with the mass conservation, it follows that
$$\int_{t_1}^{t_2}  \int_{\RN} | \nabla u|^p dx dt \le \int_{\RN} u^2(x,t_1) (x) dx \le c(p,N) M^{\frac{p\alpha}{N} +1} {t_1}^{-\alpha} .$$
The same estimate holds in terms of $u_\lambda$:
$$ \int_{t_1}^{t_2}  \int_{\RN} | \nabla u_\lambda (x,t)|^p dx dt \le \int_{\RN} {u_\lambda}^2(x,t_1) (x) dx \le c(p,N) \|u_0\|_{L^1}^{\frac{p\alpha}{N} +1} {t_1}^{-\alpha} .$$
Summing up, using H\"older inequality, it follows that the family $(u_\lambda)_\lambda$ is uniformly bounded in the space
\begin{equation*}
  u_\lambda \in L^q((t_1,t_2): W^{1,p}(\RN)) \quad \text{for all }1\le q \le p, \, 0<t_1<t_2<\infty.
\end{equation*}
The Benilan-Crandall estimate \eqref{BenilanCrandall} for $u_\lambda$ gives us that  $(\partial_t {u_{\lambda}}) \in L^1((t_1,t_2),L^1(\RN))$. Also, one can use a more precise
estimate on $u_t$ from \cite{BIV}.

First, we prove compactness of the family $(u_{\lambda})$ on bounded domains  by using the Aubin-Lions-Simon compactness criteria, see Simon \cite{simon}, that we recall next:
\begin{lemma}Let $X\subset B\subset Y$ with compact embedding $X\subset B$.  Let $\mathcal{F}$ be a bounded family of functions in $L^p(0,T:X)$, where $1 \leq p <\infty$ and $\partial \mathcal{F}/\partial t=\{\partial f/\partial
t: f \in \mathcal{F}\}$ be bounded in  $L^1(0,T:Y)$. Then the family $\mathcal{F}$ is relatively compact in $L^p(0,T:B)$.
\end{lemma}
Let $B_{\rho}:=\{ |x|<\rho\} $ for $\rho >0.$  We use the previous lemma  in the context:
$u_{\lambda} \in L^1((t_1,t_2):W^{1,p}(B_{\rho}))$, $ \partial_t u_{\lambda} \in L^1((t_1,t_2):L^1(B_{\rho}))$, and thus
 $$X= W^{1,p}_{\text{loc}}(B_{\rho}) \subseteq B =  L^1(B_{\rho}) \subseteq Y= L^1(B_{\rho}) $$
 with compact embedding $X \subset B$ .
  We conclude there exists $\mathcal{U}(x,t) \in L^1((t_1,t_2),L^1(B_{\rho})) $  such that, up to a subsequence $(u_{\lambda_j})_j$,
\begin{equation*}
  \|u_{\lambda_j}(x,t)- \mathcal{U}(x,t)\|_{L^1((t_1,t_2),L^1(B_{\rho}))}\to 0 \quad \text{as } \lambda_j \to \infty .
\end{equation*}
This analysis can be performed for an increasing sequence of balls $B_{\rho}=:B_{1} \subset B_2 \subset \dots \subset B_n \subset....$ such that $B_n \to \RN$ as $n\to \infty.$
Strong convergence in $L^1(B_2)$ implies strong convergence in $L^1(B_1)$ and thus the limit $\mathcal{U}_2$ at step 2 coincides with $\mathcal{U}$ when restricted to $B_1$.
Thus, using a diagonal argument, we can define $\mathcal{U}:\RN \times (t_1,t_2) \to \R $ as a pointwise limit in all its domain and in $L^1(K)$ sense for every bounded set $K$:
\begin{align}
u_{\lambda}(x,t) & \to \mathcal{U}(x,t) \quad \text{a.e. in }  \RN \times (t_1,t_2),  \label{conv:ulUae}\\
  \|u_{\lambda}(x,t)& - \mathcal{U}(x,t)\|_{L^1((t_1,t_2),L^1(K))}\to 0 \quad \text{as } \lambda \to \infty . \label{conv:ulU}
\end{align}

Step 3. \textit{Passage to the limit.} We prove that $\mathcal{U}$ is a weak solution to equation $\mathcal{U}_t =\Delta_p \mathcal{U}$. Firstly, notice that  pointwise convergence
\eqref{conv:ulUae} implies that $\mathcal{U}\ge 0.$  Moreover, smoothing effect \eqref{smoothing:ulambda} gives that
\begin{equation}\label{smoothing:U}
|\mathcal{U}(x,t) | \le c(p,N) M^{p\alpha/N} t^{-\alpha}.
\end{equation}
Thus convergence \eqref{conv:ulU} also holds in $L^2((t_1,t_2),L^2(K))$ for bounded $K$ and for every $0<t_1<t_2<\infty$:
\begin{equation}\label{conv:ulUL2}
  \|u_{\lambda}(x,t) - \mathcal{U}(x,t)\|_{L^2((t_1,t_2),L^2(K))}\to 0 \quad \text{as } \lambda \to \infty.
\end{equation}
 The weak formulation for $u_{\lambda}$ is the following:
$$ \int_{\tau}^{\infty} \int_{\RN} u_{\lambda}(x,t) \varphi_t(x,t) dx  - \int_{\tau}^{\infty}\int_{\RN} |\nabla u_{\lambda}(x,t)|^{p-2} \nabla u_{\lambda}(x,t) \nabla \varphi(x,t)
dx dt =0$$
for all test functions $\varphi \in C_c^{\infty}(\RN \times (\tau,\infty))$ with $\tau>0.$
Now we pass to the limit $\lambda \to \infty$. Due to the compact support of the test functions, all the integrals are in fact on bounded domains. Thus convergence
\eqref{conv:ulUL2} guarantees the first integral converges to the corresponding one for $\mathcal{U}$. For the second integral we notice that, if for instance $\text{supp}(\varphi)
\subset B_\rho(0)\times (\tau,\tau_1)$ then
$$ \int_{\tau}^{\tau_1}\int_{B_\rho}  \Big| |\nabla u_{\lambda}(x,t)|^{p-2} \nabla u_{\lambda}(x,t)\Big|^2 dx dt\le  C(\rho) \int_{\tau}^{\tau_1} \left( \int_{B_\rho}  | \nabla
 u_{\lambda}(x,t)|^{p} dx  \right)^{\frac{2(p-1)}{p}} dt <\infty,
$$
 and, thus, $|\nabla u_{\lambda}(x,t)|^{p-2} \partial_{x_i} u_{\lambda}(x,t) \to  w_i $ weakly in $L^2((\tau,\tau_1):  L^2(B_\rho))$. It follows that\\
  $(w_1,\dots, w_n)=\nabla
 \mathcal{U}.$ We conclude that
$$ \int_{\tau}^{\infty} \int_{\RN} \mathcal{U}(x,t) \varphi_t(x,t) dx - \int_{\tau}^{\infty}\int_{\RN} |\nabla \mathcal{U}(x,t)|^{p-2} \nabla\mathcal{U}(x,t) \nabla \varphi(x,t) dx
dt =0.$$

We also need to prove a suitable tail control for $u_\lambda(x,t)$. Let $\phi\in C^{\infty}(\RN)$ be a nondecreasing function such that $\phi(x) =0$ if $|x| < 1$ and $\phi(x) = 1$
if $|x| > 2$. Now we take $\phi_R  (x) := \phi(x/R)$ and then:
\begin{align*}
 \int_{|x|>2R} u_{\lambda}(x,t) dx &\le \int_{\RN} u_{\lambda}(x,t) \phi_R(x) dx  =\int_{t_1}^{t_2}
  \int_{\RN} (u_{\lambda})_t(x,t) \phi_R(x) dx dt\\
  & = -\int_{t_1}^{t_2}
  \int_{\RN}  |\nabla u_{\lambda}(x,t)|^{p-2} \nabla u_{\lambda}(x,t) \nabla \phi_R(x) dx dt.
\end{align*}
Thus, using the energy estimates, we get
\begin{align*}
\left| \int_{|x|>2R} u_{\lambda}(x,t) dx \right|\le \frac{1}{R}  \int_{t_1}^{t_2}
  \int_{\RN}  |\nabla u_{\lambda}(x,t)|^{p-1}  | (\nabla \phi(\cdot)) ( x/R)| dx dt
  < \frac{1}{R}C(u_0).
\end{align*}
This guarantees that
\begin{equation*}
  \|u_{\lambda}(x,t)- \mathcal{U}(x,t)\|_{L^1((t_1,t_2),L^1(\RN))}\to 0 \quad \text{as } \lambda \to \infty .
\end{equation*}
Smoothing effect \eqref{smoothing:ulambda} and \eqref{smoothing:U} for  $u_{\lambda}$ and $\mathcal{U}$ ensure convergence holds in every \\ ${L^1((t_1,t_2),L^q(\RN))}$, for $1<q<\infty$.

Step 4.  \textit{Identifying the limit.}
We prove now that  $\mathcal{U}$ takes a Dirac delta as initial trace, in distributional sense:
$$\lim_{t\to 0} \mathcal{U}(\cdot,t) = M \delta(\cdot).$$
This can be shows multiplying $u_\lambda(x,0)$ with a compactly supported test function and then passing the $\lambda$ parameter on the test function. Thus one can see how the support of the rescaled test function shrinks to one point, $x=0$, as $t \to 0$. Indeed, let $\phi\in C_0^\infty(\RN)$. Then
$$
  \int_{\RN} u_\lambda(x,0) \phi(x) dx   = \int_{\RN} \lambda^{\alpha} u_0(\lambda^{\beta} x) \phi(x) dx =  \int_{\RN}  u_0(x)\phi(\lambda^{-\beta} x)  dx \to M \phi(0)
$$
as $\lambda \to \infty$.
The Dirac delta function is the initial data corresponding to the Barenblatt solution. Together with the uniqueness of the initial trace for the solutions to the \eqref{PLE} as proved in \cite[Thm.I.4.2]{DiBenedetto1990}, we conclude that $\mathcal{U}(x,t)=\mathcal{B}(x,t;M)$,  the Barenblatt solution with mass $M$. In particular, $\mathcal{U}$ is self similar. Now we conclude the asymptotic behavior. Let $t=1$. Then it follows that
\begin{equation}\label{convL1:ulambda}
   \int_{\RN}|\lambda^\alpha u(\lambda^\beta x, \lambda ) - \mathcal{B}(x,1;M)| dx \to 0 \quad \text{as } \lambda \to \infty.
\end{equation}
Changing variables, this is equivalent to
$$  \int_{\RN}|u(z, \lambda ) - \lambda^{-\alpha}\mathcal{B}(\lambda^{-\beta}z, 1;M)| dz  \to 0 \quad \text{as } \lambda \to \infty.$$
Since $\mathcal{U}$ is self-similar it follows that
$$  \int_{\RN}|u(z, \lambda) - \mathcal{B}(z, \lambda;M )| dz  \to 0 \quad \text{as } \lambda \to \infty.$$ By renaming $t=\lambda$ we conclude the proof of convergence \eqref{AsympBehL1}.

For the $L^\infty$ convergence we make use of the inequality \eqref{interpolation.inequality.Rn} applied to $u_\lambda(\cdot, 1)-\mathcal{B}_\lambda(\cdot,1;M)$, i.e.
$$\| u_\lambda(\cdot,1)-\mathcal{B}_\lambda(\cdot,1;M)\|_{L^\infty(\mathbb{R}^N)} \, \le \, C_{N, \nu} \, \lfloor u_\lambda(\cdot,1)-\mathcal{B}_\lambda(\cdot,1;M)\rfloor_{C^\nu(\mathbb{R}^N)}^{\frac
N{N+\,\nu}} \, \|u_\lambda(\cdot,1)-\mathcal{B}_\lambda(\cdot,1;M)\|_{\mathrm
L^1(\mathbb{R}^N)}^{\frac{\,\nu}{N+1\,\nu}}\,,$$
where $0<\nu<1$ and $\lfloor \cdot \rfloor_{C^\nu(\mathbb{R}^N)}$ is the H\"older seminorm defined in~\eqref{C-alpha-norms}. Notice that  \\
$\lfloor
u_\lambda(\cdot,1)\rfloor_{C^\nu(\RN)}$ is finite, as proved in Lemma \ref{Lemma:reg}, since $\|(u_0)_\lambda\|_{L^1(\RN)}=\|u_0\|_{L^1(\RN)}=M$, and  $\lfloor \mathcal B \rfloor_{C^\nu(\mathbb{R}^N)}$ is finite for any $0<\nu\le1$.
By using the $L^1$ convergence \eqref{convL1:ulambda} we obtain that $\| u_\lambda(\cdot,1)-\mathcal{B}_\lambda(\cdot,1;M)\|_{L^\infty(\mathbb{R}^N)} $ converges to zero as $\lambda\rightarrow 0$. Proceeding as in step 4
and rescaling back in $\lambda$, then renaming $\lambda$ in $t$, we find exactly~\eqref{AsympBehLinf}.
\end{proof}

\section{Optimality of the \text{$\X$} data: proof of Theorem~\ref{ghp.thm}}\label{Section:Counterexample}

In this section we discuss some properties of solutions to Problem~\eqref{PLEcauchy} when they take an initial data in  the set of nonnegative integrable functions $L^1(\RN)\setminus\{0\}$. We have the following alternative: (i) either the Global Harnack Principle holds and the initial data belongs to $\X$ or (ii) the relative error is not finite and
the initial data, and therefore the corresponding solution, belong to $\X^c$.  The following proposition makes the above statement more precise.

\begin{proposition}\label{alternative}
Let $N\ge1$ and $\frac{2N}{N+1}< p < 2$ and let $u$ be the solution to Problem \eqref{PLEcauchy} with initial $u_0 \geq 0$, $u_0 \in L^1(\RN)\setminus\{0\}$.
Then the following holds
\begin{itemize}
\item[i)]If there exists $t_\star\ge0$ such that $u(t_\star) \in \X$ then $u(t)\in \X$ for any $t\ge0$, and
\begin{equation*}
\Big\|\frac{u(x,t)}{\mathcal{B}(x,t; M)}-1\Big\|_{L^\infty(\RN)}<\infty\quad\forall\, t>0\,,
\end{equation*}
\item[ii)] If $u_0\not\in \X$ then  $u(t)\not\in\X$ for any $t>0$ and
\begin{equation}\label{nonconverge.relative.error}
\Big\|\frac{u(x,t)}{\mathcal{B}(x,t; M)}-1\Big\|_{L^\infty(\RN)}=\infty\quad\forall\, t\ge0\,,
\end{equation}
where $M=\|u_0\|_{L^1(\RN)}$.
\end{itemize}
\end{proposition}
\begin{proof}
We prove first $i)$.We already know that if $u(t_\star) \in \X$ then, by the Global Harnack Principle of Theorem \ref{ghp.thm}, we can show that for any $t>t_\star$ we have that
$u(t)\le \mathcal{B}$ where $\mathcal{B}$ is a suitable Barenblatt solution. Therefore $u(t)\in \X$ for any $t\ge t_\star$.
Let us consider the case $0\le t < t_\star$. By inequality~\eqref{herrero.pierre.infty} of Lemma~\ref{herrero.pierre.infty.lemma}, we have that, for any $R>0$, the following holds:
\begin{equation*}\begin{split}
R^\frac{1}{\beta(2-p)} \int_{\RN\setminus B_R(0)}u(x,t)dx& \le \sup_{0\le \tau \le t_\star}\,R^\frac{1}{\beta(2-p)}\,\int_{\RN\setminus B_R(0)} u(x,\tau) dx \\
& \le \kappa_1\,R^\frac{1}{\beta(2-p)} \left[ \int_{\RN\setminus B_{R/2}(0)}u(x,t_\star) dx + \left(\frac{t_\star}{R^\frac{1}{\beta}}\right)^\frac{1}{2-p}\right]\\
& \le \kappa_1 2^\frac{1}{\beta(2-p)}  \left[\|u(x,t_\star)\|_{\X} + t_\star^\frac{1}{2-p}\right]\,.
\end{split}\end{equation*}
Therefore $u(t) \in \X$ for any $0\le t < t_\star$. From the GHP it follows then that $\|\frac{u(x,t)}{\mathcal{B}(x,t; M)}\|$ is finite for any $t>0$. The proof of $i)$ is concluded.

Let us prove $ii)$. Assume, by contradiction, that there exists $t_\sharp >0$ such that $u(t_\sharp)\in \X$. Then, by $i)$ we conclude that $u_0\in \X$ which is a contradiction.
Lastly, let us prove identity~\eqref{nonconverge.relative.error}: by contradiction let us assume that there exists a $t_\sharp'>0$ such that
\[
\Big\|\frac{u(x,t_\sharp')}{\mathcal{B}(x,t_\sharp'; M)}-1\Big\|_{L^\infty(\RN)} \le C <\infty
\]
We conclude that $u(x,t_\sharp')\le (1+C)\mathcal{B}(x,t_\sharp'; M)$ for any $x\in\RN$ and so $u(t_\sharp')\in X$. Therefore, by $i)$, we have that $u_0\in \X$, a contradiction.
The proof is now concluded.
\end{proof}

As a consequence of Proposition~\ref{alternative} we are now in the position of giving the proof of Theorem~\ref{ghp.thm}.

\subsection{Proof of Theorem~\ref{ghp.thm} and further equivalences}

\begin{proof}[Proof of Theorem~\ref{ghp.thm}]
We will prove first the equivalence between $i)$ and $ii)$. We notice that inequality~\eqref{ghp.inq} is a consequence of the upper bound of Theorem~\ref{upperboundXdata} combined with the lower bound of Theorem~\ref{lower.estimate.theorem} (notice that the hypothesis $u_0\neq0$ implies the validity of Theorem~\ref{lower.estimate.theorem}). The equivalence among inequality~\eqref{ghp.inq} and the fact that $u_0\in\X$ is an easily consequence of Proposition~\ref{alternative}.
Indeed, let us assume that $u(t)$ is a solution which satisfies inequality~\eqref{ghp.inq}, then $u(t_0)\in \X$ for some $t_0>0$, and, by $i)$ of Proposition~\ref{alternative}, we have that
$u_0\in\X$. \par
\noindent  Since Theorem~\ref{convergence.relative.error.thm} affirms that $i)$ implies $iii)$, to conclude the proof we only need to show that $iii)$ implies $i)$. Let us assume that~\eqref{convergence.relative.error.limit} holds. Therefore, there exists $t_0>0$ such that $u(x, t_0)\le 2 \mathcal{B}(x, t_0; M)$ and therefore $u(t_0)\in\X$. As a consequence of $i)$ of Proposition~\eqref{alternative} we have that $u_0\in\X$ and this concludes the proof.
\end{proof}

To conclude this section we will prove an equivalent criteria to establish wether $f\in\X$ or not. We remark that the proof of such equivalence is very intriguing, since it uses the result of Theorem~\ref{ghp.thm}.
\begin{proposition}\label{JLV.condition.equivalence}
Let $N\ge 1$ and $p_c:=\frac{2N}{N+1}<p<2$. Then
\begin{equation*}
f \in \X\qquad\mbox{if and only if}\qquad\int_{B_{|x|/2}(x)}|f(y)| dy = \mathrm{O}\left(|x|^{N-\frac{p}{2-p}}\right)\,\quad\mbox{as}\,\,|x|\rightarrow +\infty.
\end{equation*}
\end{proposition}
A condition similar to $\int_{B_{|x|/2}(x)}|f(y)| dy = \mathrm{O}\left(|x|^{N-\frac{p}{2-p}}\right)\,\quad\mbox{as}\,\,|x|\rightarrow +\infty$ has been introduced by V\'{a}zquez to provide a sufficient condition for the GHP in the case of \eqref{FDE}. Lately, a similar condition has been used in~\cite{BS2020} in the case of \eqref{FDE} with Caffarelli-Kohn-Nirenbergs weights.
\begin{proof} We follow the proof of Proposition 5.1 of~\cite{BS2020}.  Assume that $f \in \X$ and let $x\in\RN$, $x\neq0$. We have the following chain of inequalities
\[
\int_{B_{\frac{|x|}{2}(x)}}|f(y)|dy\le \int_{\RN\setminus B_{\frac{|x|}{2}(0)}}|f(y)|dy \le 2^{\frac{p}{2-p}-N}\,\,\|f\|_{\mathcal{X}} \,\, |x|^{N-\frac{p}{2-p}}\,=\mathrm{O}(|x|^{N-\frac{p}{2-p}})\,,\,\mbox{as}\,\,|x|\rightarrow\infty
\]
which is exactly $\textit{ii)}$. In the above line we have used that $B_{\frac{|x|}{2}(x)}\subset \RN\setminus B_{\frac{|x|}{2}(0)}$. Assume now that $f$ satisfies $\textit{ii)}$, without loss of generality we can assume that $f\neq0$. Let $u(x,t)$ be the solution to~\eqref{PLEcauchy} with initial data $u(x,t)=|f(x)|$.  A closer inspection of the proof of Theorem~\ref{ghp.thm} shows that, for any $t_0>0$, inequality~\eqref{upperboundXdata} still holds under the hypothesis $\textit{ii)}$ on the initial datum. Indeed, it is enough to use $\textit{ii)}$ in inequality~\eqref{local.smoothing.effect} to get~\eqref{inequality.outside.0} with a slightly different constant $C(u_0)$. Therefore, we conclude that the Global Harnack Principle (condition $ii)$ of Theorem~\ref{ghp.thm}) is satisfies and, by the result of Theorem~\ref{ghp.thm}, the initial datum $u_0=|f|\in\X$. The equivalence is proven.
\end{proof}

We are finally in the position to conclude the proof of the equivalence we stated in the introduction after presenting Theorem~\ref{ghp.thm}. Indeed, in addition to Theorem~\ref{ghp.thm}, we have the following.
\begin{proposition}
Under the assumptions of Theorem~\ref{ghp.thm} the following conditions are equivalents.\par
\noindent{\rm (i- Characterization in terms of the space $\X$)}
\begin{equation*}
u_0\in\X\setminus\{0\}\qquad\mbox{that is}\qquad 0<\sup_{R>0} R^{\frac{p}{2-p}-N} \int_{\RN\setminus{B_R(0)}}  |u_0(y)| dy<+\infty
\end{equation*}
\noindent{\rm (iv- Stability of the space $\X$ along the flow) }For all $t>0$
\begin{equation*}
u(t)\in\X\setminus\{0\}\qquad\mbox{that is}\qquad 0<\sup_{R>0} R^{\frac{p}{2-p}-N} \int_{\RN\setminus{B_R(0)}}  u(y,t) dy<+\infty
\end{equation*}
\noindent{\rm (v- Characterization of $\X$ with an alternative integral condition)} For all $t\ge0$
\begin{equation*}
\int_{B_{|x|/2}(x)}|u(y,t)| dy = \mathrm{O}\left(|x|^{N-\frac{p}{2-p}}\right)\,\,\mbox{as}\,\,|x|\rightarrow\infty\,,\quad\mbox{and}\quad u(t)\neq0\,.
\end{equation*}
\end{proposition}
\begin{remark}
We notice that condition $i)$ in the above is exactly condition $i)$ of Theorem~\ref{ghp.thm}.
\end{remark}
\begin{proof}
Let us assume that $u_0\in\X\setminus\{0\}$. By $i)$ of Proposition~\ref{alternative} we have that for all $t>0$ the solution $u(t)\in\X$. By uniqueness we have that $u(t)\neq0$ and therefore the right inequality of $ii)$ is satisfied. Let us now proof that $iv)$ implies $v)$. By point $i)$ of Proposition~\ref{alternative} we have that $u(t)\in \X$ for all $t\ge0$ and, by Proposition~\ref{JLV.condition.equivalence}, $v)$ is satisfied. To conclude the proof we observe the result of Proposition~\ref{JLV.condition.equivalence} affirms that $v)$ implies $i)$.
\end{proof}

\subsection{$\X$ data and the tail condition}\label{Section:tail}

We remark that the Global Harnack Principle, in the form of inequality~\eqref{ghp.inq}, was already known (at least for the Fast Diffusion Equation, see~\cite{BV2006,BS2020}) under
the stronger assumption~\eqref{decayu0} on the initial data. However, this hypothesis is non-optimal. Indeed, in what follows we give examples of functions that are in $\X$ and do
not satisfy~\eqref{decayu0}. Let $\alpha, \beta>0$  and let us define
\[
g_{\alpha, \beta}(y):=\sum_{k=2}^\infty \frac{\chi_{B_{k}^\beta}(y)}{||y|-k|^\alpha}\,,
\]
where $\chi_{B_{k}^\beta}(y)$ is the characteristic function of the set $B_k^\beta:=\{x\in\RN : k\le|x|\le k+k^{-\beta}\}$.  For any  $0<\alpha<1$ and $\beta(1-\alpha)>N$ we have
that $g_{\alpha, \beta} \in L^1(\RN)$. Indeed,

\begin{equation*}\begin{split}
\int_{R^N}g_{\alpha, \beta}dx&=\sum_{k=2}^\infty\,\int\limits_{k\le|x|\le k+k^{-\beta}}\frac{dy}{||y|-k|^\alpha}=\omega_N\,\sum_{k= 2}^\infty\int_{k}^{k+k^{-\beta}}
\frac{r^{N-1}}{|r-k|^{\alpha}}dr\,, \\
&\le 2^{N-1}\,\omega_N\,\sum_{k=2}^\infty \,k^{N-1}\,\int_{0}^{k^{-\beta}}\frac{ds}{s^{\alpha}}= \frac{2^{N-1}\,\omega_N}{1-\alpha}\, \sum_{k=2}^\infty \,
\frac{1}{k^{1+\beta(1-\alpha)-N}}\,,
\end{split}\end{equation*}
Moreover, we have that $g_{\alpha, \beta} \in \X$ under some suitable condition. Indeed we have
\[\begin{split}
\sup_{R\ge0}R^{\frac{p}{2-p}-N}\,\int_{\RN\setminus B_R(0)}g_{\alpha, \beta}(y)dy &\le \int_{\RN} |y|^{\frac{p}{2-p}-N}g_{\alpha, \beta}(y)dy \\
&\le \sum_{k=2}^\infty (2k)^{\frac{p}{2-p}-1}\,\int_{k}^{k+k^{-\beta}}\frac{1}{|r-k|^\alpha}dr\\
&=\frac{2^\frac{2(p-1)}{2-p}}{1-\alpha}\sum_{k=2}^\infty \frac{1}{k^{1+\beta(1-\alpha)-\frac{p}{2-p}}}\,,
\end{split}\]
which converges whenever $\beta(1-\alpha)>\frac{p}{2-p}$. Ultimately, the function $g_{\alpha, \beta}$ does not verify  assumption~\eqref{decayu0} since
\[\begin{split}
\limsup\limits_{|x|\rightarrow \infty} g_{\alpha, \beta}(x)|x|^\frac{p}{2-p} & \ge \limsup\limits_{n\rightarrow\infty}g_{\alpha, \beta}(n+n^{-\beta})n^\frac{p}{2-p}\\
& \ge \limsup\limits_{n\rightarrow \infty}n^{\alpha\beta+\frac{p}{2-p}}=\infty\,.
\end{split}
\]

We lastly just resume that the space $\X$ is optimal in many ways. It is the biggest space where the Global Harnack Principle holds and where any solution converges in relative
error to the Barenblatt profile. Moreover, when the initial data is in $\X$ the convergence in relative error does not hold and the solutions have a different (space)tail with the
respect to the Barenblatt profile.

\subsection{Initial data in $\X^c$}\label{Section:DataNotX}

In the last part of this section we focus on what can happen for  initial data in $\X^c$. In what follows we construct super/sub solutions to~\eqref{PLE}  which exhibit a particular tail behaviour which differ from the one of the Barenblatt solution. Let us begin by introducing the subsolution.
We postpone the proofs at the end of the Section.

\begin{proposition}\label{Prop:subsol}
Let $N\ge1$ and  $ \frac{2N}{N+1}<p<2$, , $C_2>0$ and $0< \epsilon < \min \big(\frac{1}{\beta p
(2-p)},\frac{p-1}{2-p}  \big) $. Then
\begin{equation*}
\mathcal{N}(x,t)=\frac{1}{\left(D(t)+ |x|^{\frac{p}{p-1}}\right)^{\frac{p-1}{2-p}-\epsilon}},
\end{equation*}
with
\[
D(t) = \left( C_1(p,\epsilon,N)\, t + C_2 \right)^{\frac{1}{\epsilon(2-p)}},
\]
is a subsolution to the ~\eqref{PLE} equation. The value of the constant $C(N, \varepsilon, p)>0$ is given at the end of the proof.
\end{proposition}
In a similar manner, a family of super-solutions is constructed.

\begin{proposition}\label{supersol}Let $ \frac{2N}{N+1}<p<2$, $N\ge1$, $0<\epsilon< \frac{p-1}{2-p}$ and $C_4>0$. Then
\begin{equation*}
\mathcal{R}(x,t)=\frac{G(t)^{\frac{p-1}{2-p}-\epsilon}}{\left(G(t)+ |x|^{\frac{p}{p-1}}\right)^{\frac{p-1}{2-p}-\epsilon}},
\end{equation*}
with
\[
G(t) = \left( C_3(N, \varepsilon, p)  \, t + C_4^{p-1} \right)^{\frac{1}{p-1}},
\]
is a super-solution to the \eqref{PLE} equation. The value of the constant $C_3(N, \varepsilon, p)>0$ is given at the end of the proof.
\end{proposition}
As a consequence of Propositions~\ref{Prop:subsol} and \ref{supersol} we can exhibit solutions whose spatial behaviour differs from the one of the Barenblatt solution. Indeed, we have the following result.

\begin{theorem}\label{counterexample.decay}
Let  $N\ge1$ , $ \frac{2N}{N+1}<p<2$,  and $0< \epsilon < \min \big(\frac{1}{\beta p
(2-p)},\frac{p-1}{2-p}  \big) $. Let $u(x,t)$ be the solution to~\eqref{PLEcauchy} with an initial datum $u_0$ which satisfy
\begin{equation*}
N(x,0)\le u_0(x)\le R(x,0)\qquad \forall x \in\RN
\end{equation*}
for some $C_2, C_4>0$. Then
\[
N(x,t)\le u(x,t)\le R(x,t)\qquad \forall t>0\,,\,\, x\in\RN\,.
\]
\end{theorem}
In other words,  for $|x|$ big, the decay of the solution $u(x,t)$ is given by
\[
u(x,t)\approx\frac{1}{|x|^{\frac{p}{2-p}-\varepsilon'}}\,,
\]
for some $\varepsilon'>0$.

Let us conclude this section with the proof os Propositions~\ref{Prop:subsol}, ~\ref{supersol} and Theorem~\ref{counterexample.decay}.

\begin{proof}[Proof of Proposition~\ref{Prop:subsol}]
We will make use of the formula of the $p-$Laplacian for radial functions $f(|x|)$:
$$
\Delta_p (f(|x|))= |f'(r)|^{p-2}\Big[(p-1) f''(r) +\frac{N-1}{r}f'(r) \Big], \quad r=|x|.
$$
Thus
$$\frac{\partial}{\partial t} \mathcal{N}(x,t)=- (\frac{p-1}{2-p}-\epsilon) \frac{1}{\left(D(t)+ |r|^{\frac{p}{p-1}}\right)^{\frac{p-1}{2-p}-\epsilon+1}} D'(t),
$$
and (by lengthy computations)
\begin{equation*}
\Delta_p  \mathcal{N}(x,t) = -
\left(\frac{p-1}{2-p}-\epsilon\right)^{p-1} \left( \frac{p}{p-1}\right)^{p-1}
\frac{\Big[ N D(t) -\Big(\frac{1}{\beta p (2-p)}-\epsilon \Big)  p\, r^{\frac{p}{p-1}}   \Big]}{
\left(
D(t)+ |r|^{\frac{p}{p-1}}
\right)^{ \frac{1}{2-p}-\varepsilon(p-1)}}
\end{equation*}
We take $\displaystyle 0< \epsilon < \min \big(\frac{1}{\beta p (2-p)},\frac{p-1}{2-p}  \big) $. We search for a function $D(t)$ such that $\mathcal{N}(x,t)$ is a sub-solution to
the \eqref{PLE}  equation:
$$
\frac{\partial}{\partial t} \mathcal{N}(x,t)  \le \Delta_p  \mathcal{N}(x,t) .
$$
This is equivalent to:
\begin{equation*}
  D'(t) \ge
\left(\frac{p-1}{2-p}-\epsilon\right)^{p-2} \left( \frac{p}{p-1}\right)^{p-1}
\frac{\Big[ N D(t) -\Big(\frac{1}{\beta p (2-p)}-\epsilon \Big)  p r^{\frac{p}{p-1}}   \Big]}{
\left(
D(t)+ |r|^{\frac{p}{p-1}}
\right)^{\varepsilon(2-p)}}\\
\end{equation*}
Thus, it is sufficient to take $D(t)$ such that
\begin{equation}\label{differential.inq.subsolution}
D'(t) \ge C(p,\epsilon) D(t)^{- \epsilon(2-p)  }  N D(t)\,,
\end{equation}
with $C(p,\epsilon)= \left(\frac{p-1}{2-p}-\epsilon\right)^{p-2} \left( \frac{p}{p-1}\right)^{p-1}$.
By defining
$$
D(t)  =: \left( C_1(p,\epsilon,N) t + C_2 \right)^{\frac{1}{\epsilon(2-p)}}
$$
with $C_1(p,\epsilon,N)= N \epsilon (2-p)\left(\frac{p-1}{2-p}-\epsilon\right)^{p-2} \left(
\frac{p}{p-1}\right)^{p-1}\,$
 and $C_2 = D(0)^{\epsilon(2-p)}.$ By construction $D(t)$ satisfies the differential inequality~\eqref{differential.inq.subsolution} and
therefore $N(x,t)$ is a subsolution. The proof is then concluded.
\end{proof}

\begin{proof}[Proof of Proposition~\ref{supersol}]
Arguing similarly as in Proposition \ref{Prop:subsol} we find that $\mathcal{R}$ is a super-solution if it satisfies the inequality:
\begin{align*}
G'(t)
 r^{\frac{p}{p-1}}  + \kappa(p, \varepsilon)
\frac{G(t)^{(\epsilon+1)(2-p)}}{
\left(
G(t)+ |r|^{\frac{p}{p-1}}
\right)^{ \epsilon (2-p)}}
    \cdot     \Big[ N G(t) -\Big(\frac{1}{\beta p (2-p)}-\epsilon \Big)  p \,r^{\frac{p}{p-1}}   \Big]\ge 0.
\end{align*}
with $\kappa(p, \varepsilon):=\left(\frac{p-1}{2-p}-\epsilon\right)^{p-2} \left( \frac{p}{p-1}\right)^{p-1}.$ Let $0<\epsilon< \frac{1}{\beta p (2-p)}$. Then it is sufficient to
take
\begin{equation}\label{differential.inq.supersolution}
G'(t) \ge \kappa(p, \varepsilon) \Big(\frac{1}{\beta p (2-p)}-\epsilon \Big)  p \frac{G(t)^{(\epsilon+1)(2-p)} }{
\left(
G(t)+ |r|^{\frac{p}{p-1}}
\right)^{ \epsilon (2-p)}}\,.
\end{equation}
Since the supremum of the right-hand-side of the above inequality is achieved at $|r|$=0 it is sufficient to ask
$$
G'(t) \ge \kappa(p, \varepsilon) \Big(\frac{1}{\beta p (2-p)}-\epsilon \Big)  p \, G(t)^{  (2-p)}.$$
By defining
$$G(t) := \left( C_3(p,\epsilon,N)  \, t + C_4^{p-1} \right)^{\frac{1}{p-1}},
\quad C_3(p,\epsilon,N)=(p-1)\Big(\frac{1}{\beta p (2-p)}-\epsilon \Big)  \kappa(p, \varepsilon).$$
the differential inequality~\eqref{differential.inq.supersolution} is satisfied and the proof is concluded.
\end{proof}

\begin{proof}[Proof of Theorem~\ref{counterexample.decay}]
The proof follows easily from Propositions~\ref{Prop:subsol}, and ~\ref{supersol}.
\end{proof}

\section{Global gradient decay estimates: proofs of Theorems \ref{Thm:gradient.decay} and \ref{Thm:gradient.convergence}}\label{Sec.Grad}
We begin this section with the  proof of Theorem~\ref{Thm:gradient.decay}. \par

\begin{proof}[Proof of Theorem \ref{Thm:gradient.decay}]
We recall a local estimate on the behaviour of the gradient $\nabla u$ where $u$ is a solution to~\eqref{PLEcauchy} with an integrable initial datum $u_0$.  Our proof is based on the following inequality (proven by Zhao in~\cite[Lemma 2.5, pag.621]{Zhao1995}) which holds under the running assumptions:
\begin{equation}\label{inequality.grandient}
\sup_{x\in B_R(0) } |\nabla u(x,t)| \le C\, \left[ t^{-(N+1)\beta} \left( \int_{B_{8R}(0)} u_0(x) dx \right)^{2\beta} + t^{\frac{1}{2-p}} R^{-\frac{2}{2-p}}\right]\,,\quad\mbox{for any}\quad t>0\,,\,\,R>\frac{1}{2}\,,
\end{equation}
where $C=C(N,p)>0$ is a numerical constant. We remark that in~\cite{Zhao1995} the above Lemma is stated for a $\sigma$-finite Borel measure as initial datum. We preferred to state the result in the above form, since its more general form is outside of the scope of the present paper. As well, we prefer to cite this results instead of the more complete~\cite[Theorem 1]{Zhao1995} as, it seems to us, the statement of the latter contains several misprints. We lastly remark that, taking the limit for $R\rightarrow \infty$ in~\eqref{inequality.grandient}, one finds the global \emph{smoothing effect} for the gradient, in the form of
\begin{equation}\label{global.smoothing}
\|\nabla u\|_{L^\infty(\mathbb{R}^N)} \le  C\, t^{-(N+1)\beta}\,\|u_0\|_{L^1(\mathbb{R}^N)}^{2\beta}\,.
\end{equation}
Let $0\le u_0\in L^1(\R^N)$ and let $u(x,t)$ be the weak solution to Problem \eqref{PLEcauchy}. Let us define, for any $\lambda>0$ and $y\in\R^N$,  $u_{\lambda, y}(x,t)= \lambda^\alpha u(\lambda^\beta x+y, \lambda t)$.  We notice that $u_{\lambda, y}(x,t)$ is a solution to  Problem \eqref{PLEcauchy} with initial datum $\lambda^\alpha u_0(\lambda^\beta x+y)$. Furthermore, we have that
\begin{equation*}
\nabla u_{\lambda, y}(0,1)=\lambda^{\alpha+\beta}\, \nabla u(y, \lambda) = \lambda^{(N+1)\beta}\,\nabla u(y, \lambda)\,,
\end{equation*}
where we have used that $\beta=N\,\alpha$. Let us first prove inequality~\eqref{time.gradient.decay}. By applying inequality~\eqref{inequality.grandient}, with $R=1$ and $t=1$, to the rescaled solution $u_{\lambda, y}(x,t)$ we get:
\begin{equation*}\begin{split}
|\nabla u(y, \lambda)| &= \lambda^{-(N+1)\beta}\, |\nabla u_{\lambda, y}(0,1)| \\
& \le \lambda^{-(N+1)\beta}\, C\, \left(\int_{\R^N}u_0(z)dz\right)^{2\beta}\,,
\end{split}\end{equation*}
where in the third line we just changed variables by $z=\lambda^\beta\,x+y$ and integrated on the whole $\R^N$. We remark that the above chain of inequalities is nothing else than inequality~\eqref{time.gradient.decay}.

Let us now prove inequality~\eqref{space.gradient.decay}, we recall that, in this case, the initial datum $u_0\in\X$. Let us first consider the case $|y|>1$. As before, we apply inequality~\eqref{inequality.grandient}, with $R=|y|\lambda^{-\beta}/32$ and $t=1$, to the rescaled solution $u_{\lambda, y}(x,t)$:
\begin{equation}\label{inequality.gradient.1}\begin{split}
|\nabla u(y, \lambda)| &= \lambda^{-(N+1)\beta}\, |\nabla u_{\lambda, y}(0,1)| \\
& \le \lambda^{-(N+1)\beta}\, C\, \left[\left(\lambda^{\alpha} \int_{B_{8R}(0)}u_0(\lambda^\beta\,x+y)dx\right)^{2\beta}+32^\frac{2}{2-p}\,\frac{\lambda^\frac{2\beta}{2-p}}{|y|^\frac{2}{2-p}}\right] \\
& \le \lambda^{-(N+1)\beta}\, C\, \left[\left(\int_{B_{8\,\lambda^\beta R}(y)}u_0(z)dz\right)^{2\beta}+32^\frac{2}{2-p}\,\frac{\lambda^\frac{2\beta}{2-p}}{|y|^\frac{2}{2-p}}\right]\,
\end{split}\end{equation}
where, in the last line, we just changed variables by $z=\lambda^\beta x+y$. We remark that, since $R=|y|\lambda^{-\beta}/32$, we have that $B_{8\,\lambda^\beta R}(y) \subset \RN\setminus B_{|y|/4}(0)$. Therefore, by applying definition~\eqref{X}, we deduce, from inequality~\eqref{inequality.gradient.1} that, for any $|y|>1$
\begin{equation}\label{second.estimate.proof}
|\nabla u(y, \lambda)| \le \lambda^{-(N+1)\beta}\, \kappa\, \frac{\|u_0\|_{\X}^{2\beta}+\lambda^\frac{2\beta}{2-p}}{\left(1+|y|\right)^{\frac{2}{2-p}}}\,,
\end{equation}
where we have also used the fact that $2\beta\left[p/(2-p)-N\right]=2/(2-p)$; in the above estimate, $\kappa>0$ is numerical constant. Combining inequality~\eqref{second.estimate.proof} with~\eqref{global.smoothing} we obtain~\eqref{space.gradient.decay} and the proof is concluded.
\end{proof}

The Global Harnack Principle describe, in a quantitative way, the behaviour of the solution $u$ to~\eqref{PLEcauchy}  when the initial data is taken in $\mathcal{X}_p$. In this section we address a similar question for the behaviour of  $\nabla u$. In what follows we will give the proof of Theorem~\ref{Thm:gradient.convergence}. The key argument is a very intriguing connection between the~\eqref{PLE} and the~\eqref{FDE} (see the books~\cite{Vazquez2007,JLVSmoothing})
\begin{equation*}
u_t(x,t)=\Delta u^m(x,t), \quad x\in \RN, \, t>0.
\end{equation*}
It is widely known that~\eqref{PLE} and~\eqref{FDE} enjoy several common results. In dimension $N=1$ the relation between~\eqref{PLE} and~\eqref{FDE} is the following: the derivative of the~\eqref{PLE} is a solution to~\eqref{FDE}, see~\cite{Iagar2008}. As we shall explain below, this generalizes to several dimensions when we consider radial solutions.  We recall that, for both equations, radial initial data generate radial solutions.

Let us fix some notation. In what follows we consider radial solutions defined on $\RN$, $N$ being the topological dimension. We will denote by $\overline{r}=|x|$ the coordinates
for the~\eqref{FDE} equation and by $r=|x|$ in the \eqref{PLE} case. Let us consider  $\overline{u}( \overline{r},t): \mathbb{R}^N\times(0, \infty)\rightarrow \mathbb{R}$ and let us
assume it is a radial solution to the equation:
\begin{equation}\label{fde.radial}
\overline{u}_t= \overline{r}^{1-\overline{n}} \,\frac{\partial}{\partial \overline{r}}\left(\overline{r}^{\overline{n}-1}\,|\overline{u}|^{m-1}\, \overline{u}_r\right)\,,
\end{equation}
where $\overline{n}$ is a positive parameter and $\overline{u}_r$ is the radial derivative. We notice that $\overline{u}( \overline{r},m\,t)$  is a radial solution to the \eqref{FDE} when $\overline{n}=N$.
Equation~\eqref{fde.radial} is sometimes referred to as a Weighted \eqref{FDE} with Caffarelli-Kohn-Nirenberg weights (see~\cite{Bonforte2017a, Bonforte2017, BS2020}) and can be re-written
as
\begin{equation}\label{ckn.fde.radial}
\overline{u}_t =
|x|^\gamma\,\nabla.\left(|x|^{-\gamma}\,|\overline{u}|^{m-1}\,\nabla \overline{u}\right)\,,
\end{equation}
where
\begin{equation*}
\gamma=N-\overline{n}\,.
\end{equation*}

We also notice that, for a radial solution $u(r,t)$, the~\eqref{PLE} can be rewritten as
\begin{equation}\label{ple.radial}
u_t = r^{1-n} \, \frac{\partial }{\partial r} \left(r^{n-1}\,|u_r|^{p-2}\,u_r \right)\,.
\end{equation}

We recall that, in ~\eqref{fde.radial}, the parameter $\overline{n}$ plays the role of an artificial dimension and is not, in general, an integer. It is unusual to consider
equations in a continuous dimension, however, in the radial case, this allows us to unveil some unexpected features. Indeed, the following radial equivalence has been proven in
\cite{Iagar2008}.
\begin{theorem}[{\cite[Thm.1.2]{Iagar2008}}]\label{radial.transformation.thm} Suppose $2<\overline{n}<\infty$. Then the radially symmetric solutions $u$ and $\overline{u}$ of
equations \eqref{fde.radial}, respectively \eqref{ple.radial}, are related through the following transformation:
let $ r=\overline{r}^{\frac{2m}{m+1}}$
\begin{equation}\label{radial.transformation}
\partial_r u(r,t)= D\,\overline{r}^\frac{2}{m+1}\,\overline{u}( \overline{r},t)\,, \quad D= \left(\frac{(2m)^2}{m(m+1)^2}\right)^\frac{1}{m-1}\,,
\end{equation}
where the correspondence of the parameters is
\begin{equation*}
p=m+1,\,\quad n= \frac{(\overline{n}-2)(m+1)}{2m}\,.
\end{equation*}
\end{theorem}
In \cite{Iagar2008} the authors also analyze the case $0<\overline{n}<2$, however we have decided to not report their results here since we are not going to use them.

\smallskip
\begin{proof}[Proof of Theorem \ref{Thm:gradient.convergence}]
The main ingredients of the proof are the radial transformation given in~\eqref{radial.transformation} and the convergence to the Barenblatt profile for the solutions to~\eqref{fde.radial} proven in~\cite[Theorem 3.3]{BS2020}. Let us fix the notation: we will use the variable $r$ for the solutions to
the~\eqref{PLE} and the variable $\overline{r}$ for solutions to~\eqref{fde.radial}. Observe that the transformation between the two variables is given by
$r=\overline{r}^{\frac{2m}{1+m}}$ and that $p=1+m$. Let us consider $u(x,t):\mathbb{R}^N\times(0, \infty)\rightarrow \mathbb{R}$ to be the solution to~\eqref{PLEcauchy} with initial
data $u_0$. In what follows we will sometimes denote $u(x,t)$ as $u(r,t)$.  Also, recall that $\mathcal{B}$ denotes the Barenblatt solution to the \eqref{PLE}.

Motivated by formula \eqref{radial.transformation}, we define
\begin{equation}\label{initial.datum.fde.radial}
\overline{u}_0(\overline{r}):=\frac{1}{D\,\overline{r}^\frac{2}{1+m}}\, \left(\partial_r u_0 \right)(\overline{r}^\frac{2m}{1+m})\,,
\end{equation}
where $D$ is given in~\eqref{radial.transformation}. Let  $\overline{u}(\overline{r}, t)$ to be the solution to the Cauchy problem associated to
equation~\eqref{ckn.fde.radial} with initial datum $\overline{u}_0(\overline{r})$ defined in~\eqref{initial.datum.fde.radial} with parameters
\[
\gamma=N-\overline{n}\,,\,\,\ \overline{n}= 2 \, N  \frac{p-1}{p} + 2\,\quad \mbox{and}\quad m=p-1\,.
\]
Since $u_0$ is non-increasing, then $\partial_r u_0<0$ and therefore $\overline{u}_0(\overline{r})<0$. By the Comparison Principle, it follows that  $\overline{u}<0$. As we shall see, this will not represent an issue.

Notice that, a-priori, there is no relation between $u$ and $\overline{u}$ since, Theorem \ref{radial.transformation.thm} does not specify the correspondence between the initial data. Thus, we derive some useful properties for $\overline{u}$, and then, we will be able to show that relation \eqref{radial.transformation} holds, more exactly, this will be done in Step 1.

In step 2 we prove that the solution $\overline{u}(\overline{r}, t)$ to~\eqref{ckn.fde.radial} verifies the hypothesis of Theorem 3.3 of~\cite{BS2020}. Then, we  deduce the convergence in relative error to the corresponding Barenblatt profile
\begin{equation}\label{bareblatt.profile.ckn}
\mathfrak{B}(x,t; \overline{M})= \frac{t^\frac{1}{1-m}}{\left(a_0\frac{t^{2\vartheta}}{\overline{M}^{2\vartheta(1-m)}}+a_1|x|^2\right)^\frac{1}{1-m}}\,,
\end{equation}
where
\begin{equation*}
\frac{1}{\vartheta}=\overline{n}\,\left(m-\frac{\overline{n}-2}{\overline{n}}\right)\,,
\end{equation*}
$a_0, a_1$ are positive parameters which depend on $m, N, \overline{n}$, and $\overline{M}$ is the mass of the Barenblatt profile, i.e., $\overline{M}=\int_{\R^N} \mathfrak{B}(x,t; \overline{M}) |x|^{\overline{n}-N} dx$, for any $t>0$. In this case, $\mathfrak{B}(x,t; \overline{M})$ is the fundamental solution to the \eqref{ckn.fde.radial}. Moreover, observe that  the Barenblatt profiles $\mathcal{B}$ and $-\mathfrak{B}$ are related by
\begin{equation}\label{B:Bbar}
D \, \overline{r}^{\frac{2}{m+1}} \mathfrak{B}(\overline{r},t; \kappa\,\overline{M}) =- \partial_r \mathcal{B}(r,t; \overline{M}), \quad
r=\overline{r}^{\frac{2m}{m+1}},
\end{equation}
which is in fact \eqref{radial.transformation} applied for these two solutions. We remark that the mass of $\mathfrak{B}$ is corrected by a multiplicative factor $\kappa=\kappa(N, p)>0$. A simple computation  shows that
\begin{equation}\label{kappa}
\kappa=\frac{pN}{2(p-1)D}\,.
\end{equation}

\noindent\textit{Step 1. Existence and relation between solutions}. Let us consider the Cauchy problem posed on $\R^N\times(0, \infty)$ for equation~\eqref{ckn.fde.radial} equipped with an initial datum $\overline{u}_0(\overline{r})$ defined in~\eqref{initial.datum.fde.radial}. It has been proven in~\cite[Proposition 7]{Bonforte2017a} that, if $\gamma<0$ and the initial datum $\overline{u}_0(\overline{r})\in L^1(\mathbb{R}^N, |x|^{-\gamma}dx)$, then there exists a unique solution \footnote{In Proposition 7 of~\cite{Bonforte2017a} the
authors assume that the initial datum is in $L^\infty(\mathbb{R}^N)$, however such assumption can be generalized to merely asking $u_0\in L^1(\mathbb{R}^N, |x|^{-\gamma}dx)$}. The condition on $\gamma$ amounts to verify that
\[
N-\overline{n}=N-2-2\,\frac{p-1}{p}N= \frac{(2-p)N-2p}{p}<0\,,
\]
which holds since $p >\frac{2N}{N+1}> \frac{2N}{N+2}$. On the other hand, the integrability condition $u_0\in L^1(\mathbb{R}^N, |x|^{-\gamma}dx)$ is verified if $
\int_0^\infty\overline{r}^{\overline{n}-1}\overline{u}_0(\overline{r})d\overline{r}<\infty$. To verify that this latter condition holds, we proceed as follows. Observe that, since $u_0 \in C^2(\mathbb{R}^N)$
and it is radial we necessarily have that $\frac{\partial u_0}{\partial r}(0)=0$ and $|\frac{\partial u_0}{\partial r}(\rho)|\le C\rho$ in a (right)neighborhood of the origin. Therefore,
we have that
\[
|\overline{r}^{\overline{n}-1}\overline{u}_0(\overline{r})|\le C\, \overline{r}^s\quad\mbox{where}\quad s=\frac{2(m-1)}{1+m}+2N\frac{p-1}{p}+1\,.
\]
Notice that for $p>1$, the exponent $s\ge-1$. Therefore the function $\overline{r}^{\overline{n}-1}\overline{u}_0(\overline{r})$ is integrable in a neighborhood of the origin. Now, we prove its integrability for $\overline{r}$ close to $\infty.$  Indeed,
due to the decaying assumption~\eqref{gradient.decay.hp} and by~\eqref{initial.datum.fde.radial}, close to infinity we have that
\begin{equation}\label{behaviour.infinity}
|\overline{u}_0(\overline{r})|\le \frac{C}{\overline{r}^{\frac{2}{1+m}}\,\overline{r}^{\frac{2}{2-p}\frac{2m}{1+m}}}=\frac{C}{\overline{r}^\frac{2}{1-m}}\,.
\end{equation}
By a straightforward (however lengthy) computation we have that  with
\[
|\overline{r}^{n-1}\overline{u}_0(\overline{r})|\le \frac{C}{\overline{r}^l}\quad\mbox{with}\quad l= \frac{2}{2-p}-2N\frac{(p-1)}{p}-1\,,
\]
we have that $l>1$ if $N\ge1$ and $p\ge\frac{2N}{N+1}$. We conclude that, under the running assumptions, $\overline{u}_0\in L^1(\mathbb{R}^N, |x|^{-\gamma}dx)$.  Therefore
$\overline{u}(\overline{r},t)$ exists and it is bounded and continuous (at least $C^\alpha(\mathbb{R}^N)$ for some $1\ge\alpha>0$), as it has been proven in~\cite{Bonforte2019}.

Let us first verify that the solution $\overline{u}$ is related to $u$ through the radial transformation defined in~\eqref{radial.transformation}. To this aim, we
define, for any $t\ge0$ the function
\begin{equation}\label{defn:v}
v(r, t):=D\,\int_{r}^\infty \overline{u}(\overline{\rho}^\frac{p}{2(p-1)}, t) \rho^\frac{1}{p-1}d\rho\,,
\end{equation}
where $D$ is as in~\eqref{radial.transformation}. Observe that  definition \eqref{defn:v} is well posed, since, for $t>0$, by the Global Harnack Principle for $\overline{u}$, we have that
\[ \left|
\overline{u}(\overline{\rho}^\frac{p}{2(p-1)})\,\rho^\frac{1}{p-1}  \right|\le
\frac{C\,\rho^\frac{1}{p-1}}{\rho^{\frac{2}{1-m}\frac{p}{2(p-1)}}}=\frac{C}{\rho^{\frac{2}{2-p}}}\quad\mbox{and}\quad \frac{2}{2-p}>1\quad\mbox{since}\quad p>1\,.
\]
On the other hand, for $t=0$ we have, by using~\eqref{initial.datum.fde.radial}, that $v(r,0)=u(r,0)$. Thanks to the result of Theorem~\ref{radial.transformation.thm} we have that
$v(r,t)$ is a solution to~\eqref{PLEcauchy} with initial datum $u(r,0)$, since \eqref{defn:v} is the integrated version of \eqref{radial.transformation}. By uniqueness we conclude that $u(r, t)=v(r,t)$.\par

\medskip

\noindent\textit{Step 2: Uniform convergence in relative error. Proof of the limit~\eqref{gradient.decay.inq}}. To establish the uniform convergence in relative error (\cite[Theorem 3.3]{BS2020}) of the solution $\overline{u}(\overline{r}, t)$ to the Barenblatt profile  $-\mathfrak{B}(x,t; \overline{M})$ defined in~\eqref{bareblatt.profile.ckn} we need to verify that
$$1>m>m_c(\gamma)=\frac{N-2-\gamma}{N-\gamma}=\frac{\overline{n}-2}{\overline{n}}$$ and that the initial datum $|\overline{u}_0(\overline{r})|\le C\, \overline{r}^{-\frac{2}{1-m}}$
for some $\overline{r}\ge R_0$\, (notice that this is exactly~\eqref{behaviour.infinity}). We observe that the condition $m$ amounts to verify that
\[
p^2-p\frac{3N+1}{N+1}+\frac{2N}{N+1}=\left(p-1\right)\left(p-\frac{2N}{N+1}\right)>0\,.
\]
Such a condition is satisfied since $p > \frac{2N}{N+1}\ge1$. Recall that $\overline{u}(\overline{r})$. Applying Theorem 3.3 of~\cite{BS2020} (its proof can be generalized to non-positive solutions). Therefore, we have established that
\begin{equation*}
\left\|\frac{\overline{u}(\overline{r},t)}{-\mathfrak{B}(\overline{r},t; \overline{M})}-1\right\|_{L^\infty(\R^N)}\rightarrow 0\,\quad\mbox{as}\quad t\rightarrow +\infty\,,
\end{equation*}
for $\overline{M}$ being the mass of the (negative) initial data
\begin{equation*}
\overline{M}=-\int_{\RN}\overline{u}_0(|x|)\,|x|^{-\gamma}dx\,.
\end{equation*}
Let us now recall the relation \eqref{B:Bbar} between $\mathcal{B}(r,t; \kappa^{-1}\,\overline{M})$ and $-\mathfrak{B}(\overline{r},t; \overline{M})$, from which  we deduce that
\[
\left\|\frac{\partial_r u(r,t)}{\partial_r \mathcal{B}(r, t; \kappa^{-1}\,\overline{M})}-1\right\|_{L^\infty(\R^N)}=\left\|\frac{D\,\overline{r}^\frac{2}{1+m}\,\overline{u}(\overline{r},t)}{-D\,\overline{r}^\frac{2}{1+m}\,\mathfrak{B}(\overline{r}, t; \overline{M})}-1\right\|_{L^\infty(\R^N)}=\left\|\frac{\overline{u}(\overline{r},t)}{-\mathfrak{B}(\overline{r},t; \overline{M})}-1\right\|_{L^\infty(\R^N)}\,.
\]
This establishes that limit~\eqref{gradient.decay.inq} holds for a  profile $\mathcal{B}(r, t; \kappa^{-1}\,\overline{M})$ with mass $\kappa^{-1}\,\overline{M}$, where $\kappa$ is as in~\eqref{kappa}. It only remains to prove that $\kappa^{-1}\,\overline{M}=\int_{\mathbb{R}^N}u_0(x)dx$.  From~\eqref{gradient.decay.inq}, taking into account that $\partial_r u$ and $\partial_r \mathcal{B}$ are non-positive, we deduce that, for any $\varepsilon>0$ there exists $t_\varepsilon>0$ such that
\[
(1+\varepsilon)\, \partial_r \mathcal{B}(r, t; \kappa^{-1}\,\overline{M}) \le \partial_r u(r,t) \le (1-\varepsilon)\, \partial_r \mathcal{B}(r, t;\kappa^{-1}\, \overline{M})\quad\mbox{for any}\,\,r\ge0\,,\,\,t\ge t_\varepsilon\,,
\]\
Integrating the above inequality from $r$ to $\infty$ we find that, for any $\varepsilon>0$
 \[
(1-\varepsilon)\,\mathcal{B}(r, t; \kappa^{-1}\,\overline{M}) \le  u(r,t) \le (1+\varepsilon)\,  \mathcal{B}(r, t; \kappa^{-1}\,\overline{M})\quad\mbox{for any}\,\,r\ge0\,,\,\,t\ge t_\varepsilon\,.
\]
This is equivalent to
\[
\left\|\frac{u(r,t)}{\mathcal{B}(r, t; \kappa^{-1}\,\overline{M})}-1\right\|_{L^\infty(\R^N)}\xrightarrow[]{t\to 0} 0\,.
\]
The above limit holds only if $\kappa^{-1}\,\overline{M}=\int_{\mathbb{R}^N}u_0(x)dx$, see Theorem~\ref{convergence.relative.error.thm}. The proof is concluded.
\end{proof}

\smallskip

\section{Mass conservation}\label{sec.conservation.mass}
Mass conservation for Problem \eqref{PLEcauchy} has been proved by Fino, D\"uzg\"un and Vespri \cite{FDV} for more general equation.   For completeness, we provide here an alternative proof for the  case of the \eqref{PLEcauchy} problem.
\begin{proposition}\label{conservation.mass}{sec.conservation.mass}
Let $N\ge 1$, $\frac{2N}{N+1}< p < 2$ and let $u$ be the solution of Problem \eqref{PLEcauchy} with initial data $u_0 \geq 0$, $u_0 \in L^1(\RN)$. Then, for any $t\ge0$, we have
that
\begin{equation*}
\int_{\RN}u(x,t)dx=\int_{\RN}u_0(x)dx\,.
\end{equation*}
\end{proposition}
For the proof the previous result we need some technical lemma, that we suply in the following.
\begin{lemma}\label{gradient.estimate.lemma}
Let $1<p<2$ and $u$ be the solution of Problem \eqref{PLE}. There exists a constant $\kappa_3=\kappa_3(N, p)$ such that for any $R>0$ and for any $0\le s \le t \le T$
\begin{equation}\label{gradient.estimate.inequality}
\frac{1}{R}\int_s^t \int_{B_R(0)} |\nabla u(x,\tau)|^{p-1} dxd\tau \le \kappa_3 \, \left(\frac{t-s}{R^\frac{1}{\beta}}\right)^\frac{1}{p} \left[\int_{B_{4R}(0)}u_0(x)dx +
\left(\frac{T}{R^\frac{1}{\beta}}\right)^\frac{1}{2-p}\right]^\frac{2(p-1)}{p}\,.
\end{equation}
\end{lemma}
\begin{proof}
By using inequality~\eqref{gradient.annulus} of Lemma~\ref{Benedetto.Lemma} with $\psi$ a cut-off function such that $\psi=1$ on $B_R(0)$ and $\psi=0$ in $\RN\setminus B_{2R}(0)$. We
then obtain for any $\varepsilon=\left((t-s)/R^p\right)^{\frac{1}{2-p}}$
\begin{equation*}\begin{split}
\frac{1}{R}\int_s^t\int_{B_R(0)}& |\nabla u(x,\tau)|^{p-1}dxd\tau \\
&\le \frac{C_1}{R}  \,\int_{s}^t\int_{B_{2R}}(t-\tau)^{\frac{1}{p}-1}
\left[u(x,\tau)+\left(\frac{(t-s)}{R^p}\right)^\frac{1}{2-p}\right]^{\frac{2}{p}(p-1)}dxd\tau\,, \\
& \le \frac{C_2 (t-s)^\frac{1}{p}}{R}\, R^\frac{d(2-p)}{p}\, \left[\sup_{s \le \tau \le t}\int_{B_{2R}}u(x,\tau) +
\left(\frac{(t-s)}{R^\frac{1}{\beta}}\right)^\frac{1}{2-p}\right]^\frac{2(p-1)}{p}\, \\
& \le C_3 \left(\frac{t-s}{R^\frac{1}{\beta}}\right)^\frac{1}{p}\left[\int_{B_{4R}}u_0(x)dx + \left(\frac{t}{R^\frac{1}{\beta}}\right)^\frac{1}{2-p}\right]^\frac{2(p-1)}{p}\,,
\end{split}\end{equation*}
where in the second line we have used H\"older inequality while in the last line we have used inequality~\eqref{herrero.pierre.local} of Lemma~\ref{herrero.pierre.local.lemma}.
\end{proof}
We remark that inequality~\eqref{gradient.estimate.inequality} has already appeared in~\cite[pag. 268]{DiBenedetto1990}, however the proof used in that paper is somehow different from
what is presented here. The conservation of mass now easily follows from inequality~\eqref{gradient.estimate.inequality} of Lemma~\ref{gradient.estimate.lemma}.
\medskip

\begin{proof} (of Propostion \ref{conservation.mass}).
Let $\phi\in C^{\infty}(\RN)$ be a non-decreasing function such that $\phi(x) =0$ if $|x| < 1$ and $\phi(x) = 1$ if $|x| > 2$. Let $\phi_R  (x) := \phi(x/R)$ that is supported in
$B_{2R}(0)$ and satisfies $\phi_R=1$ in $B_R(0)$, $|\nabla \phi_R  | \le \frac{C}{R}.$
We use the weak formulation~\eqref{weak.formulation} for the test function $\phi_R$:
\begin{multline*}
\Big|\int_{\RN}u_0(x)\phi_R(x)dx-\int_{\RN}u(x,t)\phi_R(x)\,dx\Big| \le C \frac{1}{R}\int_s^t\int_{B_{2R}(0)}|\nabla u(x,\tau)|^{p-1} dx \,d\tau\\
\le   C\,\kappa_3 \left(\frac{t-s}{R^\frac{1}{\beta}}\right)^\frac{1}{p}\left[\int_{B_{4R}(0)}u_0(x)\,dx +
\left(\frac{t}{R^\frac{1}{\beta}}\right)^\frac{1}{2-p}\right]^\frac{2(p-1)}{p}\,,
  \end{multline*}
where in the last step we have used lemma~\ref{gradient.estimate.lemma}. Taking the limit for $R\rightarrow \infty$ we have the assertion.
\end{proof}

\section{Appendix}
\subsection{An interpolation lemma}
Here we recall an interpolation lemma which goes back to~\cite{Gagliardo1958} and to~\cite[p.~126]{Nirenberg1959}, see also~\cite{Nirenberg1966}. Let $f:\Omega \rightarrow \R$ be a function and let us define the H\"older seminorm
\begin{equation}\label{C-alpha-norms}
\lfloor f\rfloor_{C^\nu\left(\Omega\right)}:=\sup_{\substack{x,y\in\Omega\\x\neq y}}\frac{|f(x)-f(y)|}{|x-y|^\nu}\,.
\end{equation}
In what follows, we use the notation $\omega_N=|\mathbb S^{N-1}|=2\,\pi^{N/2}/\Gamma(N/2)$.

\begin{lemma}
Let $p\ge1$ and $\nu\in(0,1)$,  $R>0$ and $x\in\R^N$. Then there exists a positive constant $C_{N, \nu, p}$ such that for any $f\in\mathrm
L^p(B_{2R}(x))\cap C^\nu(B_{2R}(x))$
\begin{equation*}
\nrm f{\mathrm L^\infty(B_{R}(x))} \, \le \, C_{N, \nu, p} \, \left(\lfloor f\rfloor_{C^\nu(B_{2R}(x))}^{\frac N{N+p\,\nu}} \, \|f\|_{\mathrm
L^p(B_{2R}(x))}^{\frac{p\,\nu}{N+p\,\nu}} + \tfrac{\|f\|_{\mathrm L^p(B_{2R}(x))}}{R^\frac{N}p}\right)\,.
\end{equation*}
Analogously
\begin{equation}\label{interpolation.inequality.Rn}
\nrm f{\mathrm L^\infty(\R^N)} \, \le \, C_{N, \nu, p} \, \lfloor f\rfloor_{C^\nu(\R^N)}^{\frac N{N+p\,\nu}} \, \|f\|_{\mathrm
L^p(\R^N)}^{\frac{p\,\nu}{N+p\,\nu}}\quad  \text{for all } \,f\in\mathrm L^p(\R^N)\cap C^\nu(\R^N)\,,
\end{equation}
where in both cases
\begin{equation*}
C_{N, \nu, p} = 2^{\frac{(p-1)(N+p\nu)+Np}{p(N+p\nu)}}
\left(
1+\tfrac{N}{\omega_N}
\right)^\frac1{p}\,
\left( 1+ \left( \tfrac{N}{p\nu}\right)^\frac{1}{p}\right)^{\frac{N}{N+p\nu}}
\,
\left(
\left(\tfrac N{p\,\nu}\right)^{\frac{p\,\nu}{N+p\,\nu}}
+
\left(\tfrac{p\,\nu}N\right)^{\frac{ N}{N+p\,\nu}}
\right)^{1/p}\,.
\end{equation*}
\end{lemma}
The proof is done in \cite[Lemma 16, pp.36, formula (102)]{BDNS2020Suplement}.

\subsection{Regularity of solutions}\label{Appendix.Regularity} Solutions to the \eqref{PLE} problem are H\"older continuous, see \cite{DiBenedetto}. We refer also to \cite{EstVaz} when $N=1$.
We devote this section to recall some known results concerning regularity due to DiBenedetto and, moreover, to give a quantitative estimate of the $C^\nu$ seminorm of the solution in all $\RN$ in terms of the mass of the solution.

 Let us fix some notation first. We call $\Omega_T=\Omega\times(T_0, T]$ where $\Omega$ is a smooth, bounded domain in $\mathbb{R}^N$. We call $\Gamma=\partial
\Omega\times [T_0, T] \cup \Omega\times \{T_0\}$ its parabolic boundary. For any $\mathit{K}\subset\Omega_T$, let us introduce the parabolic distance following the notation from \cite{DiBenedetto}:
\[
p-\mathrm{dist}(\mathit{K}, \Omega_T):=\inf_{\substack{(x,t)\in \mathit{K}\\(y,s)\in \Gamma}} \left(\|u\|_{L^\infty(\Omega_T)}^\frac{2-p}{p}|x-y|+|s-t|^\frac{1}{p}\right).
\]

For a definition of \emph{local weak solution} to~\eqref{PLE} we refer to~\cite[Chapter II]{DiBenedetto}. For the purposes of this proof we only need to know that the above theorem applies to solutions to Problem~\eqref{PLEcauchy}.

\begin{theorem}\label{holder.regularity.thm}\cite[Thm.1.1, Ch.IV]{DiBenedetto} Let $u$ be a bounded local weak solution to~\eqref{PLE}. Then $u$ is locally H\"older
continuous in $\Omega_T$, and there exists constants $\gamma>1$, $\nu>0$, depending only on $N, p$, such that, for any $\mathit{K}\subset\Omega_t$
\begin{equation*}
|u(x_1, t_1)-u(x_2, t_2)|\le \gamma \|u\|_{L^\infty(\Omega_T)} \left(\frac{\|u\|_{L^\infty(\Omega_T)}^{\frac{2-p}{p}}|x_1-x_2|+|t_1-t_2|^\frac{1}{p}}{p-\mathrm{dist}(\mathit{K},
\Omega_T)}\right)^\nu\,,
\end{equation*}
for any $(x_1, t_1), (x_2, t_2)\in \mathrm{K}$. In particular $\gamma$ and $\nu$ do not depend on $\|u\|_{L^\infty(\Omega_T)}$.
\end{theorem}

\begin{lemma}\label{Lemma:reg}Let $u$ be the solution to Problem \eqref{PLEcauchy}. Let $\gamma>0$ and $\nu>0$ given by Thm.\ref{holder.regularity.thm}. Then
$$ \lfloor u(\cdot,1) \rfloor_{C^\nu(\mathbb{R}^N)}\le  C(p, N) \|u_0\|_{L^1(\mathbb{R}^N)}^{p\beta}\,.$$
\end{lemma}
\begin{proof}

For any $\tau>0$, let us define  the function
\[
u^\tau(x,t):=\tau^\frac{p}{2-p}\,u(\tau\,x,  t)\,,
\]
we notice that $u^1=u$. One can easily check that $u^\tau$ is a solution to~\eqref{PLEcauchy} with initial datum
$\tau^\frac{p}{2-p}\,u_0( \tau\, x)$. Let us define as well  $\Omega_T=\big(B_8\setminus B_{1/4} \big) \times (1/4, 4]$ and $\mathrm{K}=\big( B_4\setminus
B_{1/2}\big) \times (1/2, 2]$ where $B_r$ is the ball of radius $r$ centered at the origin. Observe that
\[
p-\mathrm{dist}(\mathit{K}, \Omega_T) \ge   \|u\|_{L^{\infty}(\Omega_T)}^{\frac{2-p}{p}}\, \inf_{\substack{x\in K\\ y \in \Gamma}}|x-y| =: \frac{1}{4}\|u\|_{L^{\infty}(\Omega_T)}^{\frac{2-p}{p}}\, \,,
\]
By applying Theorem~\eqref{holder.regularity.thm} we find that
\begin{equation*}\begin{split}
\lfloor u(\cdot, 1)\rfloor_{C^\nu(B_{4\tau}\setminus B_{\tau/2})}&=\frac{\lfloor u^\tau(\cdot,1)\rfloor_{C^\nu(B_4\setminus B_{1/2})}}{\tau^{\nu+\frac{p}{2-p}}}
\le \gamma  \frac{\|u^\tau(\cdot,1)\|_{L^\infty(B_8\setminus B_{1/4})}^{\frac{2-p}{p}\nu+1}  }{\tau^{\nu+\frac{p}{2-p}}} \frac{1}{ \left(p-\mathrm{dist}(\mathit{K},
\Omega_T)\right)^\nu} \\[2mm]
&\le \gamma 2^{2\nu}\, \frac{\|u^\tau\|_{L^\infty(B_8\setminus B_{1/4} \times (1/4, 4])}}{\tau^{\nu+\frac{p}{2-p}}}\le
\gamma 2^{2\nu}\,\frac{\|u\|_{L^\infty(\mathbb{R}^N \times (1/4, 4])}}{\tau^\nu}\,.
\end{split}
\end{equation*}
Similarly, one finds that, for some $0<\tilde{k}<\infty$
\[
\lfloor u(\cdot, 1)\rfloor_{C^\nu(B_4)}\le \frac{\gamma}{\tilde{k}^\nu}\, \|u\|_{L^\infty(\mathbb{R}^N \times (1/4, 4])}\,.
\]
Lastly, we observe that, by the smoothing effect inequality~\eqref{smoothingEffect}, we have that
\[
\|u(x,t)\|_{L^\infty(\mathbb{R}^N \times (1/4, 4]))} \le C(p, N) 4^{N\beta}\|u_0\|_{L^1(\mathbb{R}^N)}^{p\beta}\,.
\]
To estimate the $C^\nu$ norm of $u(x,1)$ over $\mathbb{R}^N$ we proceed as follows: let $x, y\in\mathbb{R}^N$,then either $|x-y|< 1$ or $|x-y|\ge 1$. In the latter case we
have that $$|u(x,1)-u(y,1)|\le 2\|u(\cdot,1)\|_{L^\infty(\mathbb{R}^N)}|x-y|^\nu.$$
Let us consider now the case $|x-y|<1$. We have that, either $|x|\le 2$ or there
exists $j\ge1$ such that $2^j<|x|\le 2^{j+1}$. In the former case, we have that $x, y \in B_4$ and thus
\[
|u(x,1)-u(y,1)|\le \lfloor u(\cdot, 1)\rfloor_{C^\nu(B_4)}\,|x-y|^{\nu }\,.
\]
In the latter, we have the following chain of inequalities
$$
2^{j-1}< 2^j-1\le |x|-|y-x|\le |y|\le |x|+|y-x|\le 2^{j+1}+1\le 2^{j+2},
$$
 thus $x, y \in B_{ 2^{j+2}}\setminus
B_{2^{j-1}}$. It follows that
\[
|u(x,1)-u(y,1)| \le \lfloor u(\cdot, 1)\rfloor_{C^\nu(B_{ 2^{j+2}}\setminus B_{2^{j-1}})}\,|x-y|^{\nu}\,.
\]
Combining the above inequalities, we find that
\[\begin{split}
\lfloor u(\cdot,1) \rfloor_{C^\nu(\mathbb{R}^N)} &\le \max\left\{2\|u(\cdot,1)\|_{L^\infty(\mathbb{R}^N)}\,,\lfloor u(\cdot, 1)\rfloor_{C^\nu(B_4)}\,,\sup_{j\ge1}\lfloor
u(\cdot, 1)\rfloor_{C^\nu(B_{ 2^{j+2}}\setminus B_{2^{j-1}})}\right\} \\
& \le 2\, \gamma  \max\{1,2^{2\nu}\,, \tilde{k}^{-\nu}\}  \, C(p, N) 4^{N\beta}\|u_0\|_{L^1(\mathbb{R}^N)}^{p\beta}\,.
\end{split}\]
The proof is now concluded.
\end{proof}

\subsection{Some technical lemmata}

In this section we are going to state a technical lemma, which will be widely used in what follows. Let us introduce the following notation, for any $0\le r \le R$ we define the
annulus $A(r, R)$ to be
\begin{equation*}
A(r, R):= \{x \in \RN : r\le|x|\le R\}\,.
\end{equation*}
We remark that the value $r=0$ is allowed and in that case $A(0, R)=B_R(0)$ for any $R\ge0$.

\smallskip
The following estimate is a generalization of inequality I.4.3 contained in Lemma I.4.1 of  \cite[pag. 240]{DiBenedetto1990}. The main difference with the result contained in
\cite{DiBenedetto1990} is that we allow the test function to be supported in an annulus rather than in a ball.
\begin{lemma}\label{Benedetto.Lemma}
Let $N\ge 1$,  $1<p<2$ and $u$ be the solution of Problem \eqref{PLE}. There exists a constant $\kappa=\kappa(N, p)$ such that for any $\varepsilon >0$,
 $0\le s\le t$, $0\le r \le R$ and for any smooth function $\psi(x)$ supported in $A(r, R)$ such that $|\nabla \psi|\le K$ the following inequality holds
\begin{equation}\label{gradient.annulus}\begin{split}
\int_{s}^t\int_{A(r,R)} & |\nabla u(x,\tau)|^{p-1} \psi(x)^{p-1} dx\,d\tau \\
& \qquad \le \kappa\, \left(1+\frac{t-s}{\varepsilon^{2-p}}K^p\right)^\frac{p-1}{p}\,\int_{s}^t\int_{A(r,R)} (t-\tau)^{\frac{1}{p}-1}(u(x,\tau)+\varepsilon)^{\frac{2}{p}(p-1)}\,
dx\, d\tau  \,.
\end{split}\end{equation}
\end{lemma}

The proof of the above lemma is very similar to the one contained in~\cite{DiBenedetto1990} therefore we have decided to not include it the present paper.

In what follows we prove the following kind of "Herrero-Pierre" formula for the mass at infinity.
\begin{lemma}\label{herrero.pierre.infty.lemma}
Let $N\ge 1$, $\frac{2N}{N+1}<p<2$ and $u$ be the solution of Problem \eqref{PLE}. There exists a constant $\kappa_1=\kappa_1(N, p)$ such that for any $R>0$ and for any $T>0$
\begin{equation}\label{herrero.pierre.infty}
\sup_{0\le\tau\le T}\int_{\RN\setminus B_{2R}(0)}u(x, \tau)dx \le \kappa_1 \left[\int_{\RN\setminus B_R(0)}u(x,T)dx + \left(\frac{T}{R^\frac{1}{\beta}}\right)^\frac{1}{2-p}\right]\,.
\end{equation}
\end{lemma}
\begin{proof}
In what follows, we will denote by $B_R$ the ball of radius $R$ centered in the origin $B_R=B_R(0)$.  For any integer $k\ge1 $ let us define
\begin{equation*}
R_k:=2R - R \sum_{i=1}^k2^{-i}=\,\big( 1+\frac{1}{2^k} \big) R\quad\mbox{and}\quad \overline{R}_k:=\frac{R_k+R_{k+1}}{2}=\big( 1+\frac{3}{2^{k+2}} \big) R,
\end{equation*}
such that both $R_k\,,\overline{R}_k \searrow R$. We observe that   $ R_{k+1}<  \overline{R}_k< R_k $. Lastly let us recall that $A(\overline{R}_k, R_k)=\{ \overline{R}_k\le |x| \le
R_k \}$.

Let us define for any $k\ge1$ the function $x\rightarrow \psi_k(x)$ being a nonnegative smooth function such that
\begin{equation*}
\psi_k=0\quad \mbox{for}\quad|x|\le\overline{R}_k\,,\qquad \psi_k=1\quad \mbox{ for}\quad|x|\ge R_k\,\quad\mbox{and}\quad |\nabla \psi_k|\le c \frac{2^{k+2}}{R}\quad \forall
x\in\mathbb{R}^N.
\end{equation*}
We remark that $\psi_k$ can be obtained as a limit of smooth and compactly supported functions. Therefore $\psi_k$ can be used in the weak formulation~\eqref{weak.formulation} for
the Cauchy problem \eqref{PLEcauchy}, provided that all the integrals in~\eqref{weak.formulation} make sense, which is the case.

By testing the equation against $\psi_k$ (as in~\eqref{weak.formulation}), we obtain for any $0\le s \le T$ and for any $k\ge1$
\begin{equation}\label{first.step}
M_k \le \int_{\RN\setminus B_R} u(x, T)dx +  c\frac{2^{k+2}}{R}\int_0^T\int_{A(\overline{R}_k, R_k)}|\nabla u(x,\tau)|^{p-1} dx\,d\tau\,.
\end{equation}
where for any $k\ge1$ we define
\begin{equation*}
M_k:=\sup_{0\le \tau \le T} \int_{\RN\setminus B_{R_k}} u(x, \tau) dx\,.
\end{equation*}

Let us define, for any $k\ge1$, a new family of compactly supportes test functions $\phi_k$ such that
\[
\phi_k=1\quad\mbox{in}\quad A(\overline{R}_k, R_k)\,,\quad\phi_k=0\quad\mbox{in}\quad \RN\setminus A_k:= \RN\setminus A \big(R_{k+1}, R_k+\big(\overline{R}_k-R_{k+1}\big) \big)
\]
and
\[
\,|\nabla \phi_k| \le \frac{c}{\overline{R}_k-R_{k+1}}= \frac{c \,2^{k+2}}{R}\,.
\]
By applying inequality~\eqref{gradient.annulus} of Lemma~\ref{Benedetto.Lemma} with $\psi=\phi_k$ and $\displaystyle \varepsilon = \left(\frac{T}{R^p}\right)^{\frac{1}{2-p}}\,$ we
obtain
\begin{equation}\label{estimate.gradient}\begin{split}
&\int_s^T  \int_{A(\overline{R}_k, R_k)} |\nabla u(x,\tau)|^{p-1} dxd\tau \\
&\le \kappa
 \,\left(1+\frac{T}{\varepsilon^{2-p}(\overline{R}_k-R_{k+1})}\right)^{\frac{p-1}{p}}
 \int_0^T\int_{A_k}\left(T-\tau\right)^{\frac{1}{p}-1}\left(u+\varepsilon\right)^{\frac{2}{p}(p-1)} dx\, d\tau \\
& \le C_1 2^{p(k+2)}\left[|A_k| \left(\frac{T}{R^p}\right)^{\frac{2(p-1)}{p(2-p)}}\int_0^T(T-\tau)^{\frac{1}{p}-1}d\tau + T^{\frac{1}{p}} \sup_{0\le \tau \le
t}\int_{A_k}u^{\frac{2}{p}(p-1)}dx\right] \\
& \le C_2 2^{p(k+2)} \left[R^N\,T^\frac{1}{p}\, \left(\frac{T}{R^p}\right)^{\frac{2(p-1)}{p(2-p)}} + T^{\frac{1}{p}}\,R^{N\frac{2-p}{p}}\sup_{0\le \tau \le T}\left(\int_{A_k}u(x,
\tau) dx\right)^\frac{2(p-1)}{p}\right]\,,
\end{split}\end{equation}
where we have used the inequality $(a+b)^\alpha\le2^{\alpha}(a^\alpha+b^\alpha)$ (which holds for any $\alpha\ge0$ and any $a,b\ge0$) and the fact that $|A_k|\le \kappa_{N} R^N$
where $\kappa_N$ is constant which depends only on the dimension $N$. Combining inequality~\eqref{first.step}
with~\eqref{estimate.gradient} and using
\[
\int_{A_k}u\,dx \le \int_{\RN\setminus B_{R_{k+1}}}u\,dx\,,
\]
we conclude that for any $k\ge1$
\begin{equation}\label{second.step}
M_{k} \le \int_{\RN\setminus B_R} u(x, T)dx + C_3 2^{pk} \left[\left(\frac{T}{R^\frac{1}{\beta}}\right)^\frac{1}{2-p} + \left(\frac{T}{R^\frac{1}{\beta}}\right)^\frac{1}{p}
\,M_{k+1}^\frac{2(p-1)}{p}\right]\,,
\end{equation}
where $C_3$ depends only on $N$ and on $p$. Fix $\delta \in \left(0,1\right)$ to be chosen later, by Young inequality we have that
\[
C_3 2^{pk}\left(\frac{T}{R^\frac{1}{\beta}}\right)^\frac{1}{p} \,M_k^\frac{2(p-1)}{p} \le \delta M_{k+1} + C(N, p, \delta) 2^{\frac{p^2\,k}{(2-p)}}
\left(\frac{T}{R^\frac{1}{\beta}}\right)^\frac{1}{2-p}\,,
\]
where $C(N,p,\delta)= \left(C_3^p\,\delta^{2(p-1)}\right)^\frac{1}{2-p}$. Combining the above formula, with inequality~\eqref{second.step} we obtain for any $k\ge1$
\begin{equation*}
M_k \le \delta\, M_{k+1} + C(N,p, \delta)\,2^{\frac{p^2\,k}{(2-p)}} \left[\int_{\RN\setminus B_R} u(x, T)dx + \left(\frac{T}{R^\frac{1}{\beta}}\right)^\frac{1}{2-p}\right]\,.
\end{equation*}
Call $Z= \left[\int_{\RN\setminus B_R} u(x, T)dx + \left(\frac{T}{R^\frac{1}{\beta}}\right)^\frac{1}{2-p}\right]$, then we have the following iterating process
\begin{equation*}\begin{split}
\sup_{0\le\tau\le T}\int_{\RN\setminus B_{2R}}u(x, \tau)dx \le M_1 & \le \delta\, M_{2} + C(N,p, \delta)\,2^{\frac{p^2}{(2-p)}} Z \\
& \le \delta^2\,M_3 + Z \, C(N,p, \delta)\,\left(2^{\frac{p^2}{(2-p)}} + 2^{\frac{p^2}{(2-p)}}\,\delta\right) \\
& \le \delta^3\,M_4 + Z \, C(N,p, \delta)\,2^{\frac{p^2}{(2-p)}} \left(1 + \delta + \delta^2\,2^{\frac{2p^2}{(2-p)}}\right) \\
& \le \delta^{k}\,M_{k+1} + Z \, C(N,p, \delta)\,2^{\frac{p^2}{(2-p)}}\, \sum_{i=0}^{k-1} (\delta\,2^{\frac{2p^2}{(2-p)}})^i\,,
\end{split}\end{equation*}
choosing $\delta=2^{-\frac{2p^2}{(2-p)}-1}$ and taking the limit for $k \rightarrow \infty$ we obtain
\[
\sup_{0\le\tau\le T}\int_{\RN\setminus B_{2R}}u(x, \tau)dx  \le \kappa  \left[\int_{\RN\setminus B_R} u(x, T)dx + \left(\frac{T}{R^\frac{1}{\beta}}\right)^\frac{1}{2-p}\right]\,,
\]
which is what we wanted to prove.
\end{proof}
The above proof follows the strategy of \cite[Lemma III3.1]{DiBenedetto1990}, which reads
\begin{lemma}\label{herrero.pierre.local.lemma}
Let $1<p<2$ and $u$ be the solution of Problem \eqref{PLE}. There exists a constant $\kappa_2=\kappa_2(N, p)$ such that for any $R>0$ and for any $T>0$
\begin{equation}\label{herrero.pierre.local}
\sup_{0\le\tau\le T}\int_{B_{R}(0)}u(x, \tau)dx \le \kappa_2\,\left[\int_{B_{2R}(0)}u(x,T)dx + \left(\frac{T}{R^\frac{1}{\beta}}\right)^\frac{1}{2-p}\right]\,.
\end{equation}
\end{lemma}
The proof of the above lemma follows the line of the proof of Lemma~\ref{herrero.pierre.infty.lemma}, the only difference is that we shall implement the entire procedure in balls
$B_{R_k}$ instead of domains of type $\RN\setminus B_{R_k}$.

\smallskip\noindent {\sl\small\copyright~2021 by the authors. This paper may be reproduced, in its entirety, for non-commercial purposes.}

\end{document}